\documentclass[a4paper,12pt,reqno]{amsart} 

\usepackage{fullpage}
\usepackage{times}
\usepackage{amsfonts}
\usepackage{amssymb,amsmath}
\usepackage{enumitem}
\usepackage{amsmath}
\usepackage{mathrsfs}
\usepackage{amsthm}
\usepackage{dsfont}
\usepackage{soul}
\usepackage{stmaryrd}

\usepackage{cases}
\usepackage{graphicx} 
\usepackage{calc}
\usepackage[colorlinks=true,linkcolor=black,citecolor=magenta]{hyperref}
\usepackage[dvipsnames]{xcolor}
\usepackage{xifthen} 
\usepackage{appendix}                    
\DeclareMathAlphabet{\mathpzc}{OT1}{pzc}{m}{it}

\newcommand{\fC}{\ensuremath{\mathfrak C}} 
\newcommand{\fD}{\ensuremath{\mathfrak D}}

\newcommand{\fW}{\ensuremath{\mathfrak W}}

\newcommand{\fd}{\ensuremath{\mathfrak d}}

\newcommand{\fh}{\ensuremath{\mathfrak h}}

\newcommand{\fn}{\ensuremath{\mathfrak n}}

\newcommand{\fs}{\ensuremath{\mathfrak s}} 
\newcommand{\ft}{\ensuremath{\mathfrak t}} 
\newcommand{\fu}{\ensuremath{\mathfrak u}} 
 
\newcommand{\fw}{\ensuremath{\mathfrak w}}

\newcommand{\cA}{\ensuremath{\mathcal A}} 
\newcommand{\cB}{\ensuremath{\mathcal B}} 
\newcommand{\cC}{\ensuremath{\mathcal C}} 
\newcommand{\cD}{\ensuremath{\mathcal D}} 
\newcommand{\cE}{\ensuremath{\mathcal E}} 
\newcommand{\cF}{\ensuremath{\mathcal F}} 
\newcommand{\cG}{\ensuremath{\mathcal G}}

\newcommand{\cL}{\ensuremath{\mathcal L}}

\newcommand{\cP}{\ensuremath{\mathcal P}}

\newcommand{\cT}{\ensuremath{\mathcal T}} 
\newcommand{\cU}{\ensuremath{\mathcal U}} 
 
\newcommand{\cW}{\ensuremath{\mathcal W}} 
\newcommand{\cX}{\ensuremath{\mathcal X}}

\newcommand{\bbE}{{\ensuremath{\mathbb E}}} 
 
\newcommand{\bbG}{{\ensuremath{\mathbb G}}}

\newcommand{\bbP}{{\ensuremath{\mathbb P}}}

\newcommand{\sC}{{\ensuremath{\mathscr C}}} 
\newcommand{\sH}{{\ensuremath{\mathscr H}}}
\newcommand{\sR}{{\ensuremath{\mathscr R}}} 
\newcommand{\sG}{{\ensuremath{\mathscr G}}} 
\newcommand{\sB}{{\ensuremath{\mathscr B}}}

\newcommand{\R}{\mathbb{R}}

\newcommand{\T}{\mathbb{T}}

\newcommand{\N}{\mathbb{N}}
\newcommand{\dd}{{\rm d}}

\newcommand{\ee}{{\rm e}} 
\newcommand{\eps}{\varepsilon}

\newcommand{\VT}[1]{\|#1\|_{\operatorname{TV}}}
\newcommand{\Esp}[1]{\mathbb{E}\left(#1\right)}
\newcommand{\Prob}[1]{\mathbb{P}\left(#1\right)}

\newcommand{\partt}{\partial_t }

\newcommand{\ind}{\mathds{1}}

\newcommand{\tGM}{\textsc{GM }}

\newcommand{\sGlauber}{{\scriptscriptstyle \rm G}}

\newcommand{\sMIS}{{\scriptscriptstyle \rm MIS}}

\newtheorem{theorem}{Theorem}[section]
\newtheorem{Lemma}{Lemma}[section]

\newtheorem{Proposition}{Proposition}[section]

\newtheorem{definition}{Definition}[section]
\newtheorem{remark}{Remark}[section]

\subjclass[2020]{Primary 60K35; Secondary 82C22, 60J27,37A25}
\keywords{Interacting particle systems, cutoff phenomenon, mixing time, information percolation, Glauber dynamic, misanthrope dynamic}

\title{Cutoff Phenomenon for the Glauber-Misanthrope process in a perturbative regime}
\author{Ons Rameh }
\address{Ons Rameh: Sorbonne Université and  Université Paris Cité, CNRS, LJLL, F-75005 Paris, France 
}
\email{\href{mailto:orameh@dma.ens.fr}{orameh@dma.ens.fr}}
\date{}
\begin{document}

\begin{abstract}
    We investigate the Glauber--Misanthrope process, obtained by superposing a { Misanthrope dynamics} with finite site capacity and a Glauber dynamics which creates or annihilates particles one at a time at each site. In a perturbative regime around the diffusively scaled partial exclusion process, we prove the cutoff phenomenon and identify the limiting mixing time.  

Our approach relies on the information percolation framework developed by Lubetzky and Sly \cite{LS16,LS_exposition,univ_ising_cutoff} and extended by Tran \cite{HQ}, which we adapt to accommodate multiple occupancy per site and account for the two distinct correlation mechanisms generated by the Misanthrope and Glauber dynamics. Our lower bound on the mixing time holds under assumptions more general than those required for the matching upper bound.
\end{abstract}
\maketitle
\tableofcontents

\section{Introduction}

   The Misanthrope process is a class of interacting particle systems that models the evolution of indistinguishable particles on a lattice. Particles are permitted to share the same site, jumping at a rate that depends only on the occupancy of the departure and arrival sites. It was originally introduced by Cocozza-Thivent \cite{misanthrope} to model ions subject to both an external electric field and the repulsive forces they exert on one another.  More recently, it has also been interpreted as a traffic model \cite{kanai2010twolane,tordeux2010etude}. Moreover, we emphasize that this class is quite general, as it includes the exclusion process, the partial exclusion process, and the zero-range process, and it can be further generalized to Mass Migration processes \cite{FGD_mass_migration}. The hydrodynamic and hydrostatic limits have been thoroughly studied across various settings, including the finite-dimensional torus, random environments, and systems with open boundaries (see, for instance, \cite{kipnis1999sli, G_lecture_note_IHP, Rezakhanlou1991HydrodynamicLF, Rezakhanlou1995MicroscopicSO, BGRS_euler, BGRT_constructive_euler_review, eymard_roussignom_tordeux, bahadoran2011hydrodynamics}). Furthermore, several works have focused on the mixing behavior to quantify the speed of relaxation and mixing to equilibrium (see, for example, \cite{Wilson_2004, Lacoin2013MixingTA, Lacoin_2016, Lacoin_Leblond, lacoin2016simpleexclusionprocesscircle, salez_universality_cutoff, MS19, hermon2018cutoffmeanfieldzerorangeprocess, tran2022meanfieldzerorangeprocessunbounded, Tsegaye, benjamini2002mixingtimesbiasedcard, Levin_2016,arXiv:gu_yang}).

Another celebrated process is the Glauber dynamics, where particles are created or annihilated at each site at a rate determined by the occupancy on their neighborhood. Originally introduced by Glauber \cite{Glauber}, this model is among the most extensively studied in statistical physics. Given that the literature is vast, we refer the reader to the lecture notes by Martinelli \cite{Martinelli}, as well as to \cite{HQ} and \cite[Subsection 1.1]{LS_Ising_polynomial_time} for a summary of more recent developments. The first proofs of the cutoff phenomenon\footnote{The cutoff phenomenon is to be understood in the sense of Aldous and Diaconis \cite{Diaconis_Aldous_86}. We refer the reader to \eqref{eqdef: cutoff}} for Glauber processes were established on the complete graph \cite{DLP_glauber_complete_graph,LLP_glaubet_complete} and then, on the finite--dimensional torus \cite{Lubetzky_Sly,LS_cutoff_general_spin} at high temperature.  Later, in their seminal work \cite{LS16}, Lubetzky and Sly introduced the \emph{information percolation} method, which allowed them to prove the cutoff for Glauber dynamics in the entire high--temperature regime. This framework has since proven particularly powerful for studying the mixing behavior of various interacting particle systems (see, e.g., \cite{LS_exposition, LS_Ising_polynomial_time, Nam_sly, salez_universality_cutoff, HQ}).\\

In this article, we investigate the superposition of Misanthrope and Glauber dynamics. This type of combined model was first introduced by De Masi, Ferrari, and Lebowitz \cite{de_masi_ferrari_lebowitz}, who studied its hydrodynamic limits and non-equilibrium fluctuations for Glauber Exclusion process. More recently, the zero-range Glauber process was analyzed in \cite{KFHS_mean_curvature_ZRP_Glauber,Xu_Zhao,funaki_tsunoda,funaki}. Regarding the mixing behavior, to the best of our knowledge, the first bounds on the mixing time were given for Glauber Exclusion processes at the low temperature \cite{tsunoda} and \cite{tanaka_tsunoda} at the high temperature. Most recently, the cutoff phenomenon was proven in the full high-temperature regime by Tran in \cite{HQ}, where he used the information percolation framework, attractivity, and coupling arguments. {However, extending these results beyond exclusion dynamics remains a challenging problem, as multiple occupancy destroys several structural properties underlying existing proofs.}\\

\noindent \textbf{\emph{Our contribution.}} { In this work, we study the Glauber--Misanthrope process with a finite site capacity $\kappa$, obtained by combining the Misanthrope process with Glauber updates. We work \emph{in a perturbative regime} around the partial exclusion process \eqref{eq: perturbative :( } and investigate the mixing behavior of the resulting dynamics.

Our main result, stated in Theorems \ref{thm: Lower bound general} and \ref{thm: Upper bound}, establishes the cutoff phenomenon when the refresh updates are sufficiently frequent. The upper bound further requires the diffusive scaling $\beta_N=N^2$ \eqref{cond: trivial cluster size}, whereas the lower bound is proved under a more general assumptions \eqref{cond: condition general cluster size}.

Our proof adapts the information-percolation framework of Lubetzky and Sly \cite{LS_exposition,LS16,univ_ising_cutoff} and Tran \cite{HQ} to the Glauber--Misanthrope setting. To this end, we unfold the torus into an expanded graph following \cite{Tsegaye} that encodes the multiple-occupancy structure and keep track in this geometry of the two distinct correlation mechanisms generated by the Misanthrope dynamics and the Glauber dynamics.
}

\subsection{Our Model}
Let $b: \N \times \N \rightarrow [0,\infty)$ be a bounded function on $\N \times \N$ verifying $ b(0,\cdot)=  0 $ and we denote by $\T_N$ the one-dimensional torus. A misanthrope process on $\T_N$ is the Markov process on $\N^{\T_N}$ described by the generator 
$$
\cA_{\sMIS} J(\eta)= \sum_{x,y=1}^N  b(\eta(x), \eta(y))\left[J\left(\eta^{x, y}\right)-J(\eta)\right]\,,
$$
for local functions $J: \N^{\T_N}\rightarrow \mathbb{R}$, where $\eta^{x, y}$ is the new configuration resulting from $\eta$ after a  single particle jumps from $x$ to $y$. 

In this article, we focus on a subclass that authorises only configurations with at most $\kappa$ particles, where $\kappa \in \N^\star$. Indeed, we assume
\begin{equation}
    \tag{M0} \forall m \in \llbracket 0, \kappa \rrbracket, \, \,  b(0,m)=  0 \quad \text{and} \quad b(m, \kappa)=0\,.
\end{equation}
\textbf{\emph{{ The partial exclusion process.}}} A particular example of interest is the partial exclusion process (denoted by SEP($\kappa$)), where the jumps occur at a rate $$b(n,m)= n (\kappa-m)\,.$$
For $\kappa=1$, the model coincides with the so-called simple exclusion process. As we shall see in Subsection \ref{The expanded process},  SEP($\kappa$) coincides for any $\kappa \in \N$ with an exclusion process on an expanded graph. 

{SEP($\kappa$) was first introduced by Schütz and Sandow in \cite{schutz-sandow}. It was then investigated in \cite{carinci_giardina_GR_duality_for_stochastic,CGR_exact_formula,floreani_redig_sau_orthogonal_duality} for its duality properties and in \cite{Bea_FRancheschini_G_HL,franceschini2023nonequilibriumfluctuationssepalphaopen,Chen_Sau_higher_HL,FLoreani_redig_sau_hydrodynamcis_sep_re,Schiavo_Portinale_Sau_HL_duality}, for its hydrodynamic limit and non-equilibrium fluctuations in various settings. More recently, the mixing time of the partial exclusion process was investigated on the segment in \cite{Tsegaye}.

Let us emphasize that SEP($\kappa$) inherits many nice properties from the exclusion process. For example, it is a self-dual and \emph{gradient} model. Moreover, it is reversible with respect to a product of Binomial distribution $\otimes_{x \in \T_N}\cB\left(\kappa,\frac{\alpha}{\kappa}\right)$, for any fixed $\alpha \in [0,\kappa]$.}\\

In the sequel, we consider \emph{a perturbative regime} around the  SEP($\kappa$):
\begin{equation}
    \tag{M1}
    \label{eq: perturbative :( }
    b(n,m)=\beta_N n(\kappa- m) +  n(\kappa- m)\widetilde{b}(n,m)
\end{equation}
where $\frac{1}{\beta_N}$ represents the strength of the perturbation. We also assume that
\begin{equation}
    \label{eq: attractivity of the misanthrope}
    \tag{M2}
    \begin{split}
        \forall n \in\llbracket 0, \kappa \rrbracket, \quad m \mapsto \tilde b(n,m) \text{ is non-increasing,} \\
        \forall m \in\llbracket 0, \kappa \rrbracket, \quad n \mapsto \tilde b(n,m) \text{ is non-decreasing\,.}
    \end{split}
\end{equation}
This guarantee the attractivity\footnote{{The assumption are stronger to guarentee the attractivity of the expanded process defined in Subsection \ref{The expanded process}}} of the process (see for example \cite{Gobron_saada}). We denote by $\bar b= \max \tilde b$.\\

\noindent \textbf{\emph{ The Glauber Dynamic}} We shall consider systems in contact with a reservoir on each site by adding a Glauber dynamic on top of the misanthrope process. Although there are different choices for the reservoirs, we allow the creation and annihilation of only a single particle at a time. This follows for example the construction in \cite{Bea_FRancheschini_G_HL,KFHS_mean_curvature_ZRP_Glauber}. \\

Let $c_+:\llbracket 0, \kappa \rrbracket^{\T_N} \rightarrow (0, \infty]$ be an non-decreasing rate function and $c_-:\llbracket 0, \kappa \rrbracket^{\T_N} \rightarrow (0, \infty]$ be a non-increasing rate function. We consider the Glauber  dynamic described by the following generator
\begin{align*}
    \cA_{\sGlauber}J(\eta)&= \sum_{x=1}^N  c_+(\psi_x\eta) (\kappa-\eta(x)) \left(J(\eta^{x,+})-J(\eta)\right) + \sum_{x=1}^N c_-(\psi_x\eta) \eta(x) \left(J(\eta^{x,-})-J(\eta)\right)
\end{align*}
where $\eta^{x,+}$ (resp. $\eta^{x,-}$) is the configuration obtained by adding (resp killing) a particle at the position $x$, {$\psi_x \eta= \eta_{x+\cdot}$ is the configuration translated by $x$} and $J: {\llbracket 0, \kappa \rrbracket}^{\T_N}\mapsto \R$  is a test function. 

To avoid long-range correlations, we assume throughout this article that the rate functions $c_+$ and $c_-$ are local i.e
\begin{equation}
    \tag{G1}
    \label{eq: assumption locality}
   \text{There exists } m \in \N \text{ such that, }\, \,  c_+ \text{ and }c_- \text{ depend only on } (\eta(x)) _{|x| \leq m}\,.
\end{equation}

Moreover, Lemma \cite[Lemma 8]{HQ} ensures that   we can decompose the rate function $c_+$ and $c_-$, such that there exists  $(\lambda_j)_{j=q_--1}^{q_++1} \in \R_+^{q_++q_-+1}$
\begin{equation}
    \label{eq: decomposition of c}
    c_+(\eta) = \sum_{j=1}^{q_+ +1}  \lambda_{j} {f}_{j} (\eta) \quad \text{and} \quad c_-(\eta) = \sum_{j=1}^{q_-- 1}  \lambda_{j} {f}_{j} (\eta)
\end{equation}
where $(f_{j})_{j=1}^{q_+}:\llbracket 0, \kappa \rrbracket^{\T_N} \rightarrow \{0,1\}$  non-decreasing functions,   $(f_{j})_{j=-1}^{-q_-}:\llbracket 0, \kappa \rrbracket^{\T_N} \rightarrow \{0,1\}$  non-increasing functions and where   ${f}_{q_++1}(\cdot)= {f}_{q_--1}(\cdot) =1$,  $\lambda_{q_++1}= \min c_+$ and $\lambda_{-q_--1}= \min c_-$. \\

\noindent \emph{\textbf{ Our model.}}
In this article, we are interested in the juxtaposition of the two previous processes:
\[
\cA_{\rm GM} = \cA_{\sMIS} +  \cA_{\sGlauber} .
\]
We refer to it as the Glauber-Misanthrope process (${\rm GM}$). It is an irreducible Markov process on $\llbracket 0, \kappa \rrbracket^{\T_N}$, and therefore admits a unique invariant probability measure that we will denote by $\pi_N$. We emphasise that there is no simple expression of this invariant measure and that the process is generally not reversible with respect to $\pi_N$. \\

The classical theory of Markov processes ensures that for any initial configuration $\eta_0 \in \llbracket 0, \kappa \rrbracket ^ {\T_N}$, 
$$
P_t^ {\eta_0} \xrightarrow[N \rightarrow \infty]{} \pi_N\,, 
$$
where $P_t^ {\eta_0}$ is the law of GM started at $\eta_0$. Our aim is to quantify how fast the system mixes. With this aim, for a given of precision $\theta \in (0,1)$, we define the mixing time
 as 
$$
T_{\operatorname{mix}}^{N}(\theta):=\inf \left\{t \geq 0 ; \quad d_{N}(t) \leq \theta\right\}\,,
$$
where we denoted by $d_{N}$ the worst-case total-variation distance to equilibrium, i.e
$$
d_{N}(t)=\max _{\eta \in \llbracket 0, \kappa \rrbracket ^ {\T_N} }\left\|P_t^\eta-\pi_{N}\right\|_{\mathrm{TV}}\,.
$$
It turns out that for many Markov processes, we observe a remarkable phase transition. The distance to equilibrium remains close to one for a long period before suddenly dropping to zero over a much shorter time window. In these cases, we say that the system exhibits a Cutoff phenomenon. {This phenomenon was first discovered in the context of card shuffling \cite{Diaconis1981-yy, Aldous1986SHUFFLINGCA} and is described by}
\begin{equation}
    \label{eqdef: cutoff}
    \forall \theta \in(0,1), \quad \frac{T_{\mathrm{mix}}^{N, k}(\theta)}{T_{\mathrm{mix}}^{N, k}(1-\theta)} \xrightarrow[N \rightarrow+\infty]{ } 1 .
\end{equation}
For further details on the cutoff phenomenon, we refer to \cite[Chapter 18]{LP}. {As mentioned before, under suitable conditions, we establish in this article that \tGM exhibits a cutoff phenomenon.}

\subsection{Main Results}
Let us first develop some intuition before stating the main results. Thanks to the monotonicity of $c_+$ and $c_-$ and \eqref{eq: attractivity of the misanthrope}, the process GM is attractive. That is, for any $\eta \leq \eta^\prime$, we can couple two Glauber-Misanthrope processes $(\eta_t)_ {t\geq 0}$  and $(\eta_t^\prime)_ {t\geq 0}$  such that almost surely
\begin{equation}
    \label{eq: attractivity}
    \forall t \geq 0, \qquad \eta _ t \leq \eta_t^\prime\,.
\end{equation}
Hence, the extremal initial configurations 
\begin{equation}
    \label{eq: bot top}
    \bot = (0)_{x \in \T_N}  \qquad \text{and} \qquad \top = (1)_{x \in \T_N}
\end{equation}
are the only two possibilities for the configuration furthest from equilibrium. Moreover, the quantity 
\begin{equation}
    \label{eqdef: mathpzc r intro misanthrope}
    \mathpzc r _ t = \sum_{x=1}^N \top_t(x) - \bot_t(x) 
\end{equation}
accounts for the maximal total number of discrepancies between these two extremal processes. Heuristically, if $\mathpzc r _ t = 0$ with high probability, the processes should have already mixed. As we shall see {at the beginning of Section \ref{section: Proof the theorems - Plan of the rest of the chapter}}, it is actually enough to wait until the expected number of discrepancies reaches the fluctuation regime.\\

Following this discussion, we define for any $\eta_0 \in \llbracket 0, \kappa \rrbracket ^{\T_N}$, the discrete density profile by
\[
\forall t \geq 0,\forall x \in \T_N\qquad \alpha_ t^N(\eta_0,x) := \bbE_{\eta_0} (\eta_t(x))
\]
By abuse of notation, we will drop the dependence on $x$ if the initial configuration is translation-invariant. As explained above, we define the predicted equilibrium time as 
\begin{equation}
        \label{eq: prediction of the mixing time}
    T^N_{eq}(\eps)= \inf\left\{t\geq 0, \qquad |\alpha_t^N(\top) - \alpha_t^N(\bot) | \leq {\eps}{\sqrt{N}}
      \right\}, \qquad \forall \varepsilon>0\,.
    \end{equation}

The first main result ensures that this predictive time represents a lower bound for the mixing time.  Let us first denote by
\begin{equation}
    \label{def mu+- mu -}
    \upsilon_+= \sum_{j=- q_-}^{q_+} \lambda_j, \qquad  \upsilon_-= \lambda_{-q_- -1} +\lambda_{q_+ +1}  \qquad  \text{and} \qquad  \bar \rho= \frac{\lambda_{q_++1} }{\upsilon_-}
\end{equation}
\begin{theorem}[Lower bound]
    \label{thm: Lower bound general}
    Assume that either
    \begin{equation}
        \tag{L1}\label{cond: trivial cluster size}\beta_N=N^2 \qquad \text{and} \quad 2(2 \kappa -1)\kappa \bar{b} + (\kappa (2m+1)-1) \upsilon_+ < \upsilon_-\,, \end{equation}
    or \begin{equation}
         \tag{L2}\label{cond: condition general cluster size}
        2 (4 \kappa-2)\kappa   \bar{b} +   (4\kappa m +2 \kappa -2) (2m+1)\kappa \upsilon_+ < \upsilon_-\,.
    \end{equation}
    We have 
    $$
    \forall \theta \in (0,1),\qquad \liminf_{N \rightarrow \infty} \frac{T_{\text {mix }}^{N}(\theta)}{T^N_{eq}(1)} \geq  1 \,.
    $$
\end{theorem}

We stress that the choice of $\theta=1$ is not relevant. Indeed, we shall  prove later (see Lemma \ref{lemma: control of REDs exp control}) that for any $\theta,\theta^\prime \in \R_+^\star$,
\begin{equation}
    \label{cutoff of the predictive time}
    \frac{T_{\rm eq}^N(\theta)}{\log(N)}\, \text{ is bounded \quad and \quad}  \frac{T_{\rm eq}^N(\theta)}{T_{\rm eq}^N(\theta^\prime)} \underset{N \rightarrow 0}{\longrightarrow} 0\,. 
\end{equation}

The proof relies on the second moment method applied to $\mathpzc r$ within the \emph{information percolation} framework \cite{LS16,LS_exposition,univ_ising_cutoff}. With this aim, we control  the two-point correlation function $$\Esp{\eta(x)\eta(y)} -\Esp{\eta(x)}\Esp{\eta(y)}\,.$$ We achieve this by constructing a branching process called BEP, which traces backwards in time the dependency of $\eta(x)$ and $\eta(y)$ on the initial configuration. Heuristically, at rate $\upsilon_+$, the system makes a deterministic update by either adding or removing a particle ($f_{-q_--1}=f_{q_++1}=1$). In contrast, at rate $\upsilon_+$ for Glauber dynamics and at most $2\bar b$ for the misanthrope dynamics, the changes depend on additional information and can potentially create correlations. To avoid strong correlations, Assumptions \eqref{cond: condition general cluster size} and \eqref{cond: condition general cluster size} guarantee that these deterministic events occur frequently enough.

Lubetzky and  Sly developed the framework of information percolation based on the idea of exploring the past in order to prove the cutoff of the Glauber process \cite{univ_ising_cutoff,LS16,LS_exposition}. The key idea is that if only a small number of sites depend on the initial configuration, and hence that its information percolated up to time zero, then the system is essentially mixed. In our setting, the quantity $\mathpzc r$ precisely captures this number. {We mention that this idea of exploring backward in time is also linked to the coupling from the past algorithm \cite{Propp_wilson96}.}

{In this work, we adapt this framework to the Glauber--Misanthrope setting. Using the expanded graph representation introduced in \cite{Tsegaye}, we can handle the case $\kappa>1$ by representing the partial exclusion dynamics as an exclusion process on an enlarged graph. This allows us to extend the information percolation approach beyond the classical binary state space $\{0,1\}^{\T_N}$ to a setting with local state space $\llbracket 0,\kappa\rrbracket^{\T_N}$. A distinctive feature of this setting is that correlations may arise from two different sources: the Glauber updates and the ``pure'' Misanthrope dynamics. }\\

Regarding the upper-bound, we will focus on the regime $\beta_N = N^2$. 
\begin{theorem}[Upper bound]
    \label{thm: Upper bound}
    Assume \eqref{cond: trivial cluster size}. Then, the model exhibits a Cutoff phenomenon, and we have  for any $\theta \in (0,1)$, 
   $$
       \lim _{\substack{N \rightarrow \infty}} \,  \frac{T_{\text {mix }}^{N}(\theta)}{T^N_{eq}(1)} = 1 \,.
    $$
\end{theorem}
The idea of the proof is based on the techniques developed in \cite{HQ}{, adapted to the geometry of the expanded graph.} Without going into technical details, as this would require additional notation, the key intuition is that when $\mathpzc r$ is sufficiently small, one can hope to eliminate discrepancies one by one. For this matter, we take advantage of the partial exclusion dynamics to drive discrepancies toward suitably ``good'' regions\footnote{See Subsection \ref{subsection: RGB}; these ``good'' regions correspond to the blue positions.}. {The diffusive scaling allows these discrepensies to reach these ``good''  regions efficiently.} 
{We note that the lower bound in Theorem \ref{thm: Lower bound general} holds in a more general regime. This suggests that the cutoff phenomenon may persist beyond the diffusive scaling considered in Theorem \ref{thm: Upper bound}.}

Finally, in the {diffusive scaling} $\beta_N=N^2$, we determine the limit of $T^N_{eq}(1)$. Since we expect the contribution of the pure misanthrope part to become negligible compared to the partial exclusion contribution, we expect that at the macroscopic level, the density profile evolves according to
 \begin{equation}
     \label{Eq:reactdiff}
     \partt \alpha= R(\alpha)
 \end{equation}
 where 
 \begin{equation}
    \label{eq: reaction term}
    R (\alpha) = \mathbb{E}_{\cB\left(\kappa,\frac{\alpha}{\kappa}\right)}\left[ (\kappa-\eta(0)) c_+(\eta) - \eta(0) c_-(\eta)\right].
 \end{equation}
{Indeed, Equation \eqref{Eq:reactdiff} should describe the scaling limit of the partial exclusion process with a Glauber dynamics as we consider the diffusive scaling.}
\begin{Proposition}
    \label{prop: expression of predictive time}
    Assume \eqref{cond: trivial cluster size}. $R$ admits a unique zero $\alpha_\star$ to which both $\alpha^N(\wedge)$ and $\alpha^N(\vee)$ converge.
   Then, we have
   $$
       \lim _{\substack{N \rightarrow \infty}} \,  \frac{{T^N_{eq}(1)}}{\log(N)} = \frac{1}{2 R^\prime(\alpha_ \star)}\,.
    $$
\end{Proposition}

\subsection{Organisation of the paper}
The rest of this paper is organised as follows. In Section \ref{section: Preliminaries - Graphical constructions}, we introduce the main processes and the graphical constructions that will be used. In particular, the GM process on an expanded graph is defined in Subsection \ref{The expanded process} and the BEP is constructed in Subsection \ref{subsection: BEP construction}.

The proofs of Theorems \ref{thm: Lower bound general} and \ref{thm: Upper bound} are the aim of Section \ref{section: Proof the theorems - Plan of the rest of the chapter}. Section \ref{section: Density profile and correlation control} is devoted to the analysis of the density profile and the control of the two-point correlation function. In particular, Proposition \ref{prop: expression of predictive time} is prove in Subsection \ref{subsection: Density profile and correlation control}. Finally, in Section \ref{section: Proof of Propostion the control of one error and lcocalisation}, we establish the intermediate lemmas required for the proof of the upper bound.

\section{Preliminaries - Graphical constructions}

\label{section: Preliminaries - Graphical constructions}
\subsection{The expanded process}   
\label{The expanded process}
As specified above, the proof is based on tracing backwards in time the dependency $\eta_x$ at time $t$ with respect to the initial configuration. Given that the maximum number of particles per site is fixed at \(\kappa\), it is helpful to visualise the content of position \(x\) at time \(t\) as being contained in \(\kappa\) distinct boxes located at that position. It turns out that the dynamics of the partial exclusion primarily involve the exchange of these boxes, while \(\widetilde b\) necessitates revealing their contents. With this understanding, it is more convenient to examine our process as living on an expanded graph, as in \cite{Tsegaye}. \\

We denote by $\mathbb{G}_{N,\kappa}=(V_{N,\kappa},E_{N,\kappa})$ the graph where $$V_{N,\kappa}=\T_N \times [\kappa]\quad \text{and} \quad E_{N, \kappa}= \left\{ \{(x,i), (y,j)\}, \forall i, j \in [\kappa], \, \forall x,y \in \T_N, \, \operatorname{dist}_{\T_N}(x,y)=1 \right\} \,,$$ 
where $\operatorname{dist}_{\T_N}$ is the graph distance in $\T_{N}$ and $[\kappa]  := \llbracket 1, \kappa\rrbracket $. 

Moreover, we define the distance (different from the graph distance) between  $(x,i),(y,j) \in V_{N,\kappa}$ as 
\begin{equation}
    \label{def: distance on the expanded graph}
    \operatorname{dist}((x,i),(y,j))= \operatorname{dist}_{\T_N}(x,y)\,.
\end{equation}

 For any $u \in V_{N,\kappa}$, we denote by $B(u,m)$ the ball of radius $m \in \N$ centred in $u$.  We also abusively denote $B(e,m)= B(u,m)\cup B(v,m)$ for $e=\{u,v\}$. 

We associate to  any $\xi \in \{0,1\}^{\mathbb{G}_{N,\kappa}}$, the configuration  $\iota(\xi)\in \llbracket 0, \kappa \rrbracket^{\T_N}$, where
 \begin{equation}
    \label{eq: def iota projection}
    \forall x \in \T_N , \quad\iota(\xi)(x)= \sum_{j=1}^\kappa \xi(x,j)\,.
\end{equation}
{We also denote abusively $\iota((x,i))=x$. }

We define the expanded process on $\mathbb{G}_{N,\kappa}$ as a continuous-time Markov process on $\{0,1\}^{\bbG_{N,\kappa}}$ with the following generator: \begin{equation}
     \label{def-eq: generator of the exapnded process}
    \cA_{\rm Exp} = {\cA_ E}+{\cA}_ {C,L}+{\cA}_ {C,R} +{ \cA_{Gl}}+{\cA}_ R.
\end{equation}
where:
\begin{align}
     {\cA}_E J(\xi)& = \beta_N \sum_{(x,i) \sim (y,j)} \left(J(\xi^{(x,i) \leftrightarrow (y,j)})-J(\xi)\right) \tag{Exclusion exchanges}\\
     {\cA}_ {C,L} J(\xi)& =   \sum_{x,i,j} \widetilde b(\iota(\xi)_x,\iota(\xi)_{x+1}) \xi_{(x,i)} \left(J(\xi^{(x,i) \leftrightarrow (x+1,j)})-J(\xi)\right) \tag{Correlated  left exchanges} \\
     {\cA}_ {C,R} J(\xi)& =  \sum_{x,i,j} \widetilde b(\iota(\xi)_{x+1},\iota(\xi)_x, ) \xi_{(x+1,j)} \left(J(\xi^{(x,i) \leftrightarrow (x+1,j)})-J(\xi)\right) \tag{Correlated  right exchanges}  \\\label{eq:refreshupdate_genrator_expanded}
    {\cA}_ R J(\xi)& = \upsilon_- \sum_{(x,i)} \bar \rho \left(J(\xi^{(x,i) ,1})-J(\xi)\right) + (1- \bar \rho) \left(J(\xi^{(x,i) ,0})-J(\xi)\right) \tag{Refresh updates}\\
    {\cA}_ {Gl }J(\xi)& = \sum_{(x,i)} [c_-(\psi_x\iota(\xi))-\lambda_{-q_--1}] \left(J(\xi^{(x,i) ,0})-J(\xi)\right) \tag{Anhilation Glauber updates}\\ &\, \,+ \sum_{(x,i)} [c_+(\psi_x\iota(\xi))-\lambda_{q_++1}] \left(J(\xi^{(x,i) ,1})-J(\xi)\right) \,. \tag{Creation Glauber updates}
 \end{align} 
where $\xi^{u \leftrightarrow v}$ is the configuration obtained by exchanging the value of $\xi$ on these positions, $\xi^{u,1}$ (resp. $\xi^{u,0}$) is the configuration obtained by adding  (resp killing) a particle at the position $(x, i)$. 
We also denote by $\nu_{\rm eq}$ its invariant measure.
\begin{Lemma}
    \label{lemma: link expanded and misanthrope process}
  For any $\xi_0  \in \{0,1\}^{\bbG_{N, \kappa}}$, $(\iota(\xi_t))_t$ is a Misanthrope-Glauber process started at $\iota({\xi_0})$. Moreover, $\pi_N = \nu_{\rm eq}\circ \iota ^{-1}$.
\end{Lemma}
The proof is an immediate consequence of the expression of the generators $\cA_{\rm Exp}$ and $\cA_{\rm GM}$. As mentioned earlier, we expect the two extremal configurations to correspond to the worst-case scenarios. We denote $(\vee_t)_{t \ge 0}$ (resp. $(\wedge_t)_{t \ge 0}$) the process initialized  at the empty configuration $(0)_{u \in V_{N,\kappa}}$  (resp. the full configuration $(1)_{u \in V_{N,\kappa}}$). 
\subsection{Graphical construction of the process and its history}
\label{subsection: Graphical construction of the process and its history}
The key idea is to construct the process at time $t$ deterministically from the initial condition and a common random background process that contains the update sequence, such as the nature, time, and location of the update. 

We start by introducing the common background process. Let
$$
\Xi=\left(\left(\Xi_e^{E}\right)_{ e \in E_{N,\kappa} }, 
\left(\Xi_{e}^{C}\right)_{  e \in E_{N,\kappa} }, 
\left(\Xi_{j,u}^{G}\right)_{\substack{-q_- \leq j \leq q_+,\\ u \in V_{N,\kappa}}}, \left(\Xi_{\pm,u}^{R}\right)_{u \in V_{N,\kappa}} \right)
$$
be defined as independent Poisson processes as follows:
\begin{itemize}
    \item \textbf{Exclusion Marks:} For any $e\in E_{N,\kappa}$, $\Xi_e^{E}$ is a PPP on $\R_+$ of intensity $\beta_N$.
    \item \textbf{Refresh Marks:} For any $ u \in V_{N,\kappa}$, $\Xi_{u}^{R}$ is a  PPP on $\R_+\times \{0,1\}$ of intensity $$ \upsilon_- \dd t \otimes \cB(\bar{\rho})\,, $$ where $\bar \rho$ and $\upsilon_-$ are defined in \eqref{def mu+- mu -}.
    \item \textbf{Correlated exchange Marks:} For any $e \in E_{N,\kappa}$, $\Xi_{e}^{C}$ is a  PPP on $\R_+\times [0,1]$ of intensity $$  \bar{b} \dd t \otimes \ind_{[0,1]} dV.$$
    \item  \textbf{Glauber Marks:} For any $-q_- \leq j \leq q_+, u \in V_{N,\kappa}$,  $\Xi_{ j, u}^{G}$ is  PPP on $\R_+$ of intensity $\lambda_j$.
\end{itemize}

For any $\xi_ 0$, we can deterministically construct the process $(\xi_t)$ as a piecewise constant function changing when a mark of $\Xi$ appears. The changes follow the following rules
 \begin{itemize}
        \item If it is an exclusion mark $s$ in $e$, then we exchange the value of $\xi$ on the vertices connected by $e$.
        \item If it is a refresh mark $(s,\sigma)$ in $u$ then the value of $\xi$ at $u$ is replaced by $\sigma$.
        \item If it is a Glauber mark $s$ in $u=(x,i)$ of type $j$ then the value of $\xi$  at $u$ is replaced by $\frac{1+\operatorname{sign}(j)}{2}$ if and only if $f_{j}(\psi_x\iota(\xi))=1$.
        \item  If it is a correlated exchange mark $(s,V)$ between $u=(x,i)$ to $v=(x+1,j)$ then the value of $\xi$ at $u$ and $v$ are exchanged if $$V \leq \frac{\widetilde b( \iota(\xi_x), \iota(\xi_{x+1}))}{\bar b}  \xi_{(x,i)}+ \frac{\widetilde b( \iota(\xi_{x+1}), \iota(\xi_x))}{\bar b}  \xi_{(x+1,j)}.$$
 \end{itemize}

The main interest of a graphical construction is the tractability of the dependence in time $0$ and the created correlations. Hence, if we want to know the value of $\xi$ of site $u$ at the time $t$, we need to trace back the history and reveal only the \emph{needed} information. Indeed, the information of site $u$'s content at time $t$: either \emph{jumps} when an exclusion update occurs, \emph{branches} when a Glauber update or a correlated exchange mark happens as it requires additional information, or \emph{is killed} in case of a refresh update.\\

Following these heuristics, for any $E \subset V_{N, \kappa}, t \geq 0$, we define $\cC_t(E)$ the update history of $E$ as the collection of all marks that we encountered when we determined the content of any site $i \in E$ at time $t$. It turns out that the update history can be constructed as a Markovian branching process of information labelled by their origin and the changes evolving on $\mathbb{G}_{N,\kappa}$, and proving this will be the aim of the next subsection.

\subsection{BEP}
\label{subsection: BEP construction}
Let 
\begin{equation}
    \label{eq-def: the set of words}
    \cW= \bigcup_{i=0}^{\infty}\left( V_{N,\kappa} \times B(0, m)^i\right)
\end{equation}
be the set of words starting by an element $u \in V_{N,\kappa}$ and letters $c_1\cdots c_i$ from $\{B(0, m)\times [\kappa]\}^i$. An element $w \in \cW$ labels the datum that $u$ needed backwards in time due to a series of branching events encoded by the  $c_1\cdots c_i$. As usual, if $w^\prime$ is a prefix of $w$, $w$ is referred to as a descendant of $w ^\prime$, and it is denoted by $w\succ w^\prime$.

We will refer to an element $w \in \cW$ as the ``datum'' with label $w$, to avoid any confusion between actual particles of the original process and the particles of the BEP which models the information. Moreover, the time of the expanded process is denoted by $s$ or $t$, and we denote backward times by $\fs$ or $\ft$.\\

Given $E \subset V_{N,\kappa}$, we consider the following construction that encodes the update history: 
\begin{enumerate}
    \item At time $0$, for any $u \in E$, we put a datum labeled as $u$ in the position $u$. Each datum has independent Poissonian clocks as follows:
    \begin{itemize}
        \item \textbf{Glauber clocks}: $q_-$ independent exponential clocks of intensities $\lambda_{-1},.. ,\lambda_{-q_-}$ and $q_+$ independent exponential clocks  of intensities $\lambda_{1},.., \lambda_{q_+}$. 
        \item \textbf{Refresh clocks}: One exponential clock of rate $\upsilon_-.$ We associate to it a mark $ \mathfrak{c} \sim \cB(\bar{\rho})$ . 
    \end{itemize}
    \item Each edge $e$ has two Poisson Point Processes:
    \begin{itemize}
        \item \textbf{Exclusion marks}: One PPP on $\R_+$ of intensity $\beta_N$. 
        \item {\bf Correlated exchange marks}: One PPP on $\R_+\times [0,1]$ of intensity ${  \bar{b}}  dx \otimes \ind_{[0,1]} dv$.
    \end{itemize}
    \item A datum $w$ at the position $(x, i)$ follows the exclusion marks until one of its clocks rings, or it encounters a mark. The following then occurs:
    \begin{itemize}
        \item If it is a Refresh clock, the datum $w$ is killed. 
        \item If it is a Glauber clock, we kill $w$  and we create new data on all the sites $B(x, m)\times [\kappa]$ and they are canonically labeled as children of $w$. 
        \item If it is a correlated exchange mark of the edge $ \{(x, i),(y,j)\}$, we kill $w$, and we create new data on all the sites $\{x,y\}\times [\kappa]$ and they are canonically labeled as children of $w$. 
    \end{itemize}
    \item If a datum $w$ is created
    \begin{itemize}
        \item On an unoccupied site: the datum is referred to as a first-class datum. We associate it with its own independent clocks.
        \item On an occupied site: it is a second-class datum, and it will follow the marks of its associated first-class datum (denoted by $\mathfrak{w}(w)$) already alive on the site. Notice that it means that we may have detected a correlation. 
    \end{itemize}  
    \item For any datum $w$ alive at time $\ft$, we denote by $\ell_\ft(w)$ its position in the graph.
\end{enumerate}
{The Branching exclusion process started from $E$ refers to the set of all marks of all the data that lived between $0$ and $\ft$. } We will denote it by $(\fC_{\ft}(E))_{\ft \geq 0}$. 

We say that the datum $w$ is alive at time $\ft$ if it was born before time $\ft$ and has not been killed by time $\ft$. For any $\fs \leq \ft$, we define $\fW_{\fs}$ as the set of the first class data alive at time $\fs$ and $\mathfrak D_{\fs}$ as the set of data alive at time $\fs$. \\

\paragraph{\emph{Reveal of the content.}}
Let us fix $\ft>0$. Conditionally on $\fC_{\ft}$, given $\sigma \in \{0,1\}^{\fW_{\ft}}$, the content of each data alive in $\fW_{\ft}$, we can deterministically and inductively reveal the content of all data that lived between $0$ and $\ft$ according to the following rules:
\begin{itemize}
\item For a datum $w$ who survived until time $\ft$ has already an assigned content $\sigma(w)$
\item For a datum $w$ who died due to a refresh clock marked $\mathfrak{c}$, its content is fixed to be $\mathfrak{c}$.
\item For a datum $w$ who died due to a Glauber mark $\fs$ of type $j$, its content is fixed as:
$$
f_{j} ((\sigma(w^\prime))_{w^\prime \succ w})\,.
$$
\item For a datum $w$ who died due to a correlated exchange mark $(\fs, \mathfrak{v})$, its content depends only on 
$(\sigma(w^\prime))_{w^\prime \succ w}$. 
\end{itemize} 

Conditionally on the BEP, for any data  $w \in \bigcup_{\fs\in [0,\ft]} \fW_\fs$, there exists an function $F_{w,\ft}$ on $\{0,1\}^{\fW_{\ft}}$ such that any spin configuration $\sigma$ on $\fW_{\ft}$ satisfies
$$
\sigma_{ w }=F_{ w , {\ft}}\left(\left(\sigma_{ w ^{\prime}}\right)_{{w}^{\prime} \in \fW_{\ft}}\right)\, .
$$

The function $F_{ w, \ft}$ describes the dependence of the content of $w$ on the data still alive at time $\mathfrak t$. We emphasise that the assumption (G2-3) and (M1) imply that $F_{ w, \ft}$ is a \emph{non-decreasing}  function. 

{Let us define for any $\xi_0 \in \{0,1\}^{\bbG_{N,\kappa}}$, the configuration $\zeta_\ft(\xi_0)$ by 
\begin{equation}
    \left(\zeta_{\ft,u}(\xi_0)\right)_{u\in E}=\left(F_{ u , {\ft}}\left[\left(\xi_{0,\ell_\ft({w}^{\prime})}\right)_{{w}^{\prime} \in \fW_{\ft}}\right]\right)_{u\in E}
\end{equation}}
In other words, we chose the spins of the data in $\fW_ \ft$ according to their positions and an ``initial configuration'' \footnote{Here, we say initial configuration as it corresponds to an initial configuration in the forward direction. Notice that the procedure of revealing the contents is done in the forward direction.}.
Thanks to the homogeneity of Poisson point processes, we observe 
the following duality relation between the BEP and the expanded process.

\begin{Lemma}[Duality relation]
    \label{lemma: link between bep and expanded}
    For any $\ft>0$ and any $\xi_0 \in \{0,1\}^{\bbG_ {N,\kappa}}$, we have:
    $$
    \left(\zeta_{\ft,u}(\xi_0)\right)_{u\in E} \overset{\cL}{=} (\xi_{\ft,u})_{u\in E} \,,
    $$
    where $(\xi_t)_{\ft>0}$ is the expanded process on $\bbG_{N,\kappa}$ started at $\xi_0$. 
\end{Lemma}
{ This statement can be interpreted as a duality relation between BEP and the expanded process. Moreover, it provides a grand coupling for all expanded processes initiated from any initial configuration $\xi_0 \in \{0,1\}^{\bbG_{N,\kappa}}$, and consequently as a grand coupling for the Glauber Misanthrope process thanks to Lemma \ref{lemma: link expanded and misanthrope process}. We stress that this is only a weak form of coupling, as it holds for each fixed $\ft>0$. A stronger, trajectorial coupling is, in fact, possible, but we omit it. \\

\paragraph{\emph{Extremal configurations}}  Through this branching structure, we can construct the expanded process from extremal positions $\vee$ and $\wedge$ by using the BEP graphical construction.
\begin{equation}
    \label{eq: vee and wedge construction}
    \forall w \in V_{N, \kappa}, \, \wedge_\ft(w)= F_{w, \ft}((1)_{w^\prime \in \fW_\ft}) , \qquad \vee_\ft(w)= F_{w, \ft}((0)_{w^\prime \in \fW_\ft}).
\end{equation}

Using the monotonicity of the functions $(F_{u,\ft})_{u \in V_{N,\kappa}}$, we conclude that for any $\ft>0$,  we have a.s
\begin{equation}
    \vee_ \ft \leq \wedge_ \ft\,.
\end{equation}}

\subsection{{Locating the Information}}  Thanks to the construction of the BEP, we can easily identify the data involved in determining the content of any datum $w$.  Indeed, they correspond to the only data that appear anywhere in its history, as a descendant of $w$. 
\begin{definition}[Update History]
    \label{def: information update history}
    We define the update history for any $A \subset \bigcup_{\fs\in [0,\ft]} \fD_\fs$ as:
\begin{equation*}
    \mathscr{H}_{A}(E,\ft)= \left\{ w^\prime \in \bigcup_{\fs\in [0,\ft]} \fD_\fs; \text{ such that there exist } w \in A, \,\, w^\prime \succ w \right\}.
\end{equation*}
and we denote by ${\fh}_{A}(E,\ft)$ its cardinal. \\
For any $A\subset\bigcup_{\fs\in [0,\ft]} \fD_\fs $, we denote by $\fC_\ft(E, A)$ the set of all marks encountered by any data in the update history of any $w\in A$.
\end{definition}
Notice that we used the same notation as for the BEP, which should not cause any confusion. Indeed, by a slight abuse of notation, we identify
$$
\fC_\ft(E) = \fC_\ft(E,E) \,.
$$

Let us emphasise that the update history does not account for the potential correlation created between two different positions. Indeed, to capture this, one would also need to explore the branches beyond the first encounter with a first-class datum. 
\begin{definition}[Information Cluster]
     \label{def: information cluster}
We define the information cluster for any $A \subset \bigcup_{\fs\in [0,\ft]} \fD_\fs$ as 
\begin{equation}
    \mathscr{C}_{A}(E,\ft)= \left\{ w^\prime \in \bigcup_{\fs\in [0,\ft]} \fW_\fs; \text{ such that } \exists w \in A, \, \exists w_1 \succ w^\prime , w_2 \succ w, \text{ with } \mathfrak{w}(w_1)=\mathfrak{w}(w_2)\right\}\,.
\end{equation}
We denote by ${\fn}_{A}(E,\ft)$ its cardinal.
\end{definition}

In words, the information cluster collects all first-class data $w^\prime$ such that its descendant intersects the history of some $w \in A$. Put differently, it corresponds to all first-class data that are connected (both forward and backwards in time) to the datum $w$. 

For convenience, we will omit the dependence on $E$ in $\mathscr{H}, \mathscr{C}, \fn$ and $\fh$, when it is clear from the context.

\subsection{Categorization of the positions}
\label{subsection: RGB}
In this subsection, we consider $E=V_{N,\kappa}$ for the construction of the BEP run up to $\ft>0$ and $\xi_0 \in \{0,1\} ^{\bbG_{N,\kappa}}$. \\

Thanks to the construction of the BEP, we can distinguish the positions $V_{N,\kappa}$ according to the way their contents are mutually correlated and depend on $\xi_0$. For any position $u \in V_{N,\kappa}$, we have three possibilities:
\begin{itemize}
    \item A descendant of $u$ survived up to time $\ft$ (i.e $\fD_{\ft} \cap \mathscr{H}_u(\ft) \neq \emptyset $):  This means that in the process of revealing the content of $u$, we used $\xi_0$. Put differently, the content of $u$ depends on the initial configuration. In this case, we say that the position $u$ is \textcolor{red}{\emph{red}} and we define the set of \textcolor{red}{red positions} as
    \begin{equation}
        \label{eq: def red}
         \mathscr{R}(\ft)=\left\{v \in V_{N,\kappa}; \text{ such that } \quad  \fD_{\ft} \cap \mathscr{H}_v(\ft) \neq \emptyset \right\},
    \end{equation}
    Moreover, the monotonicity of $(F_{u,\ft})_{u\in V_{N,\kappa}}$ imply
    \begin{equation}
        \label{eq: def red with vee and wedge}
    \sR(\ft)=\{u \in V_{N,\kappa}; \text{ such that } \quad \wedge_ \ft(u) \neq \vee_\ft(u)\}
    \end{equation}
    We denote by 
    \begin{equation}
        \label{eq: def number of red}
        \mathpzc{r}_\ft=\operatorname{Card}(\mathscr{R}(\ft)) = \sum_{u \in V_{N,\kappa}} \wedge_\ft(u) -  \vee_\ft(u).
    \end{equation}
    {Thanks to Lemma \ref{lemma: link between bep and expanded}, we recover the expression defined in \eqref{eqdef: mathpzc r intro misanthrope}.}
    \item  No descendent of $u$ survived up to time $\ft$ but the history  of $u$ is not an isolated cluster (i.e $\sC_u(\ft)\neq \sH_u(\ft)$). In other words, even though the content of $u$, i.e $\zeta_{\ft,u}(\xi_0)$,  does not depend on $\xi_0$, it is correlated to the content of another position. In this case, we say that the position $u$ is  \textcolor{ForestGreen}{\emph{green}} and we define the set of \textcolor{ForestGreen}{green positions} as
    
    \begin{equation}
        \label{eq: def green}
         \mathscr{G}(\ft)=\left\{v \in V_{N,\kappa}; \text{ such that } \quad  \sC_v(\ft)\neq \sH_v(\ft)\right\},
    \end{equation}

    Let us emphasise that the green positions are more delicate to control, as they are truly correlated. Moreover, around the equilibrium, we should expect them to be predominant since the invariant measure is not a product measure. We also point out that, thanks to the grand coupling, the contents of green positions, i.e $\left(\zeta_{\ft,u}(\xi_0)\right)_{u\in \sG(\ft)}$, are \emph{the same} for any initial configuration $\xi_0$.

    \item The history of $u$ is an isolated cluster, whose history does not percolate up to the initial time (i.e $\sH_{u, \ft}=\{u\} $ and $ \{u\} \cap \fW_\ft= \emptyset $). In this case, we say that the position $u$ is \textcolor{blue}{\emph{blue}} and we define the set of \textcolor{blue}{blue positions} as for any $\fs \leq \ft$, 
    \begin{equation}
        \label{eq: def blues}
         \mathscr{B}(\fs)=\left\{v \in V_{N,\kappa}; \text{ such that } \quad  \sH_{u, \fs}=\{u\}  \text{ and } \{u\} \cap \fW_\fs= \emptyset \right\}\, .
    \end{equation}
    For $\fs =\ft$, we have 
    $$   \mathscr{B}(\fs)= V_{N,\kappa} \setminus \left[\mathscr{G}(\ft) \cup     \mathscr{R}(\ft)\right]\, .$$
    Let us observe that conditionnally on $(\sC_{u}(\ft))_ {u \in V_{N,\kappa}}$, 
    \begin{equation}
        \left(\zeta_{\ft,u}(\xi_0)\right)_{u \in \mathscr{B}(\ft)}\overset{\cL}{=}\bigotimes_{u \in \mathscr{B}(\ft)}\cB(\bar{\rho}).
    \end{equation}
    and it is independent of the other contents.
\end{itemize}

Let us denote by $\cF^{h}_\ft$ the sigma-algebra generated by the Glauber clocks and refresh clocks and exclusion and correlated exchange marks\footnote{Notice that we have not included the marks $c$ associated to the refresh clocks.} as well as $(F_{u,\ft})_ {u \in \mathscr{G}(\ft) \cup     \mathscr{R}(\ft)}$. The discussion can be summed up by the following lemma 
\begin{Lemma}
    \label{lemma: RGB conditionnally on the sigma algebra}
    Conditionnally on $\cF^{h}_\ft$, there exist $\sigma \in \{0,1\}^{\sG(\ft)}$ such that for any $\xi_0 \in \{0,1\}^{\mathbb{G}_{N,\kappa}}$, we have 
    \begin{equation}
        \label{eq: mesure RGB conditionnellement à rgb}
        (\left(\zeta_{\ft,u}(\xi_0)\right)_{u \in \mathscr{R}(\ft)},\left(\zeta_{\ft,u}(\xi_0)\right)_{u \in \mathscr{G}(\ft)},\left(\zeta_{\ft,u}(\xi_0)\right)_{u \in \mathscr{B}(\ft)} \overset{\cL}{=} \delta_{\left(\zeta_{\ft,u}(\xi_0)\right)_{u \in \mathscr{R}(\ft)}} \otimes \delta_{\sigma} \otimes \bigotimes_{u \in \mathscr{B}(\ft)}\cB(\bar{\rho}).
    \end{equation}
\end{Lemma}

We emphasise that introducing these different classes of positions (red, blue, and green), and studying their associated information clusters or update histories, is a key ingredient of the information percolation framework \cite{univ_ising_cutoff,LS_exposition,Lubetzky_Sly}.

Heuristically, if there are too many red positions, then we expect that the law significantly depends on the initial condition. Hence, we do not expect that mixing occurred. Let us give a first control in expectation. 
\begin{Lemma}[A first control of the red positions]
    \label{lemma: control of REDs exp control} Assume \eqref{cond: condition general cluster size} or \eqref{cond: trivial cluster size}.
    There exists $c_1, c_2>0$, such that for any $\fs, \ft >0$, we have
    \begin{equation}
    e^{-c_1 t} \, \Esp{\mathpzc{r}_{\fs}} \, \leq \,\Esp{\mathpzc{r}_{\ft+ \fs}}\, \leq \, e^{-c_2\ft }  \, \Esp{\mathpzc{r}_{\fs}} \,.
    \end{equation}
    where $\mathpzc r$ is defined in \eqref{eq: def number of red}.
\end{Lemma}
Notice that this implies that the predictive mixing time $T_{\rm eq}^N$ defined in \eqref{eq: prediction of the mixing time} undergoes an abrupt transition, which is the reason  why we actually have a Cutoff for this system.  Indeed, we have for any $\eps,\eps^\prime \in \R_+^\star$,
\begin{equation}
    \label{cutoff of the predictive time_2}
    \frac{T_{\rm eq}^N(\eps)}{\log(N)}\, \text{ is bounded \quad and \quad}  \frac{T_{\rm eq}^N(\eps)}{T_{\rm eq}^N(\eps^\prime)} \underset{N \rightarrow 0}{\longrightarrow} 0\,. 
\end{equation}
The same phenomenon occurs for Glauber dynamics at high temperature (see for example \cite[Claim 3.3]{univ_ising_cutoff}). The proof of this lemma is delayed to Appendix \ref{section-appendix: control of Red}.\\

On the contrary, the blue positions are very helpful as they are completely decoupled from all the others and from the initial conditions. The following lemma ensures that they do exist in a significant number and they are not localised. 

\begin{Lemma}
    \label{lemma: there is a minimum of well positionned Zs}
    For any $a>0$, there exists $c_3, c_4>0$ , such that for any  $\ft > c_{4}$, we have
    \begin{equation}
        \bbP\left[ \forall \fs \in [\ft - 1, \ft], \quad 
         \forall u \in V_{N,\kappa}, \, \operatorname{dist}(u, \sB(\fs) ) \leq c_3 \log(N) \right] \geq 1 - \frac{1}{N^a}\,.
     \end{equation}
\end{Lemma}

We omit the proof of this lemma as it is a direct adaptation of \cite[Lemma 32]{HQ}. Indeed, for any $E \subset V_{N,\kappa}$, $E \cap \sB(\ft) = \emptyset$ means that all the sites $u \in E$ are occupied by a datum created due to correlated exchange or Glauber marks or was created at $\fs=0$. Provided that $E$ is big enough, this is unlikely. The main difference here is that we need only to take into account that the correlated exchange marks which occur at rate $2 \bar b$. 
\subsection{A single red perturbation}
\label{subsec: A single red perturbation}
Let us first introduce a notation for measures on $\{0,1\}^{\bbG_{N,\kappa}}$ of the form \eqref{eq: mesure RGB conditionnellement à rgb} with a single red position. 
Given\footnote{Notice that we identified the configuration $\mathpzc{B} \in \{0,1\}^{\bbG_{N,\kappa}}$ and the set of blue positions $ \{v \in V_{N,\kappa}, \, \mathpzc{B}_v=1 \, \}$.  } 
$$\mathpzc{B}_0\in \{0,1\}^{\bbG_{N,\kappa}}, \, U_0 \in V_{N,\kappa}\setminus \mathpzc B_0\, , \zeta \in \{0,1\}^{V_{N,\kappa} \setminus [\mathpzc{B}_0 \cup \{U_0\}]} \quad \text{and}\quad  \sigma \in \{0,1\}\,, $$
 We denote by
\begin{equation}
    \label{eq: one perturbation}
    \mu(\mathpzc B_0, U, \zeta, \sigma) = \delta_{\sigma}^{U_0} \otimes \delta_{\zeta} \otimes  \cB(\bar{\rho})^{\otimes \mathpzc{B}_0} \quad \text{and} \quad \mu^\prime(\mathpzc B_0, U, \zeta) = \delta_{\zeta} \otimes \cB(\bar{\rho})^{\otimes \mathpzc{B}_0 \cup \{U_0\}}
\end{equation}
where $\delta_{\sigma}^{U_0}$ is the Dirac measure for the position $U_0$.

The key idea in \cite{HQ} is to detect to what extend the single red perturbation in $U_0$ distinguishes between two expanded processes started from $\mu(\mathpzc B_0, U, \zeta, \sigma)$ and $\mu^\prime(\mathpzc B_0, U, \zeta)$. 
 
To this end, we use the forward graphical construction of the expanded process introduced in Subsection \ref{subsection: Graphical construction of the process and its history}. We keep track of the blue positions as a process $\mathpzc{B}_t$ on $\{0,1\}^{\bbG_{N,\kappa}}$, and of the location of the red perturbation as a tracer $U_t$ taking values in $V_{N,\kappa} \cup \{\dagger\}$. {We precise that $\dagger$ is a \emph{cemetery} state as the tracer may be killed.}

We introduce the common background process.
$$
\widehat{\Xi}= \left(\left(\Xi_e^{E}\right)_{ e \in E_{N,\kappa} }, \left(\Xi_e^{H}\right)_{ e \in E_{N,\kappa} } ,
\left(\Xi_{e}^{C}\right)_{  e \in E_{N,\kappa} },
\left(\Xi_{j,u}^{G}\right)_{\substack{-q_- \leq j \leq q_+,\\ u \in V_{N,\kappa}}}, \left(\Xi_{\pm,u}^{R}\right)_{u \in V_{\N,\kappa}} \right) \,.
$$
Here, we have added the \textbf{Hesitating Exclusion Marks} $\left(\Xi_e^{H}\right)_{ e \in E_{N,\kappa} }$ where for any $e\in E_{N,\kappa}$, $\Xi_e^{E}$ is a PPP on $\R_+$ of intensity $2 \beta_N$. 

We construct the expanded process with a tracer $(\xi_t, \mathpzc{B}_t, U_t)_{t \ge 0}$ as a c\`adl\`ag, piecewise constant process driven by the marks $\widehat{\Xi}$. The evolution is given by the following rules:
 \begin{itemize}
        \item \textbf{For a refresh mark in $u$:} The configuration $\xi$ is updated as in the original dynamics. A blue particle is created at $u$ (i.e.\ $\mathpzc{B}_t(u)=1$). If $U_t = u$, then $U_t$ is sent to 
        $\dagger$.
        
        \item \textbf{For Glauber mark or a correlated exchange mark:} The configuration $\xi$ evolves as in the original dynamics. Any blue site impacted by the mark is removed (i.e.\ $\mathpzc{B}_t(u)=0$ for the corresponding sites). If $U_t$ is impacted by the mark, it is sent to $\dagger$.

        \item \textbf{For an exclusion mark on $e=(u,v)$ and $\cB(u)=0$ or $\cB(v)=0$:} The values of $\xi$ and $\mathpzc{B}$ at $u$ and $v$ are exchanged. If $U_t \in \{u,v\}$, it moves to the other endpoint.
        \item 
        \item \textbf{For an hesitating exclusion mark on $e=(u,v)$ and $\cB(u)=\cB(v)=1$:}  We sample an independent Bernoulli r.v with parameter $\frac{1}{2}$. If it is a success, we proceed as if for an exclusion mark.
 \end{itemize}

\begin{remark}[The construction of $U$] The process $U$ detects whether the single red position at $U_0$ impacted the value of $\xi_t$. A refresh mark kills the perturbation, as it erases its effect. However, a Glauber or correlated exchange mark also kills it since the corresponding content is revealed and already impacts the value $\xi_t$. 

Notice that without the hesitating marks $\Xi^H$, the position of the single red perturbation would have a deterministic, conditionally on the common background. Whereas by adding them, the position of the red perturbation is still random conditioned on  $\widehat \Xi $, thanks to the added independent Bernoulli random variables. 
\end{remark}
\begin{remark}[A version with mutltiple tracers]
    \label{remark: multiple tracers}
Notice that we are tracing the position of only one perturbation. We emphasise that we can use the same construction to trace many perturbations started from different positions which evolve independently, conditionally on $\widehat \Xi$. 
\end{remark}

Once again, this construction traces the impact of the perturbation on the value of $\xi$, similarly to the reveal content procedure of the BEP. Indeed, we observe that, conditionally on $\widehat{\Xi} $, there exist  for any $t \ge 0$ a function $\cX_t$ which \emph{does not depend on $\sigma$} such that 
\[
\xi_t = \mathcal{X}_t\big(\zeta, (\xi_{(\tau_i,u_i)})_i, (\xi_t(u))_{u \in \mathpzc B_ t}\big), \text{ and } \xi_t^\prime = \mathcal{X}_t\big(\zeta, (\xi^\prime_{(\tau_i,u_i)})_i, (\xi^\prime_t(u))_{u \in \mathpzc B_ t}\big)\,,
\]
where $(\tau_i,u_i)_{i \leq k}$ are the space-time points at which a blue site is removed due to a Glauber or correlated exchange mark and $\xi_t$ (resp. $\xi_t^\prime$) is an expanded process at $\mu(\mathpzc B_0, U_0, \zeta, \sigma)$ (resp. $\mu^\prime(\mathpzc B_0, U_0, \zeta)$).

The problem thus reduces to determining whether there exists $i \in \mathbb{N}$ such that $U_{\tau_i} = u_i$ or  if  $U_t \neq \dagger$. This point of view ensures that it is indeed sufficient to study only $\mathpzc B$ and $U$.
\begin{Lemma}[{\cite[Lemmas 37 and 39]{HQ}}]
    \label{lemma: one red control}
    Fix $t>0$. There exist a constant $c_5>0$ such that for any $\mathpzc B_0, \zeta, U_0$ as before, we have 
    \begin{equation}
        \label{eq: one red control}
        \VT{Q_t^{\mu}- Q_t^{\mu^\prime}}^2 \leq c_5 \mathbb{E}\left[ \sum_{\substack{u \in V_{N,\kappa}\,,\\s \in \Xi^G_u \cap [0,t] }} \ind_{U_{s}=V_{s}\in B(u,m)} + \sum_{\substack{e \in E_{N,\kappa}\,,\\s \in \Xi^C_e \cap [0,t] }} \ind_{U_{s}=V_{s}\in B(e,0)}+\ind_{U_{t}=V_{t}\neq \dagger} \right]\, ,
    \end{equation}
where $U$ and $V$ are iid and follow the law $U$, conditionnally on $\widehat \Xi \cap [0,t]$. 
\end{Lemma}

The proof of this lemma is based on the perturbation of a product measure lemma \cite[Proposition 3.2]{miller_peres} and is identical to the proof of Lemmas 37 and 39 in \cite{HQ}. The following lemma aims to control the right hand side of \eqref{eq: one red control}.
\begin{Lemma}
    \label{lemma: RHS- one red control}
    Let $(D_N)$ be  a sequence  such that $D_N^{2}=o(\beta_N)$ and $t>0$. There exist a constant $c_6(t)>0$ such that we have  
    \begin{equation}
        \label{eq: rhs- one red control}
      \begin{split} 
        \mathbb{E}&\left[ \sum_{\substack{u \in V_{N,\kappa}\,,\\s \in \Xi^G_u \cap [0,t] }} \ind_{U_{s}=V_{s}\in B(u,m)} + \sum_{\substack{e \in E_{N,\kappa}\,,\\s \in \Xi^C_e \cap [0,t] }} \ind_{U_{s}=V_{s}\in B(e,0)}+\ind_{U_{t}=V_{t}\neq \dagger} \right]\\ &\qquad\qquad\leq c_7 \frac{D_N^2}{\sqrt{\beta_N}}+ t
        \bbP_ {\mathpzc B_0}\left[ \exists s \in [0, t],  
        \max_{ u \in V_{N,\kappa}} \operatorname{dist}(u, \mathpzc B _s) \geq  D_N \right] \,,
      \end{split}
    \end{equation}
where $U$ and $V$ are iid and follow the law $U$, conditionnally on $\widehat \Xi \cap [0,t]$. 
\end{Lemma}

The proof of this lemma is postponed to Section \ref{section: Proof of Propostion the control of one error and lcocalisation}.

\section{Proof of the theorems}

\label{section: Proof the theorems - Plan of the rest of the  chapter}

The purpose of this section is to prove Theorems \ref{thm: Lower bound general} and \ref{thm: Upper bound}. In this section, we fix $E = V_{N,\kappa}$ and construct the BEP up to a time $t>0$ that we fix later.

As mentioned earlier, we expect the two extremal configurations $\bot$ and $\top$ (defined in \eqref{eq: bot top}) to correspond to the worst-case scenarios. Therefore, we consider the process Glauber-Misanthrope initialised from these two particular configurations. Thanks to Lemma \ref{lemma: link expanded and misanthrope process}, we obtain 

\begin{align*}
    (\bot_ t(x))_{x \in \T_N}&\overset{\cL}{=}\iota(\vee_ t)= \left(\sum_{i=1}^\kappa \vee_ t(x,i)\right)_{x \in \T_N} ,\\
    (\top_ t(x))_{x \in \T_N}&\overset{\cL}{=}\iota(\wedge_ t)= \left(\sum_{i=1}^\kappa \wedge_ t(x,i)\right)_{x \in \T_N}.
\end{align*}
where $\vee$ and $\wedge$ constructed as in \eqref{eq: vee and wedge construction} and $\iota$ defined in \eqref{eq: def iota projection}. In particular, we have, thanks to this construction 
\begin{equation*}
    \mathpzc{r}_t = \operatorname{Card}(\sR(t)) = \sum_{x =1}^N \top_t(x) - \bot_t(x)\,.
\end{equation*}

In other words, one can see that if, for some $x \in \T_N$, we have $\bot(x)=0 \neq \top(x)=1$, then there exists $u=(x, i)$ such that its update history percolates all the way back to time $t$. Naturally, if $\mathpzc{r}_t=0$, the system has already reached equilibrium. However, in that situation, we would have waited too long. In fact, as mentioned before, it turns out that mixing takes place as soon as the expected number of red sites is of the same order as its typical equilibrium fluctuations. This clearly justifies the definition of  \eqref{eq: prediction of the mixing time} 
    \begin{equation*}
  T^N_{eq}(\eps)= \inf\left\{t\geq 0, \, \alpha^N(\top) - \alpha^N(\bot)  \leq {\eps}{\sqrt{N}}\right\}=\inf\left\{t\geq 0, \, \Esp{\mathpzc{r}_t}  \leq {\eps}{\sqrt{N}}\right\}\,.
    \end{equation*}
\subsection{Lower bound} The main idea is to use the total number of particles as a distinguishing statistic and the second moment method, similarly to \cite{HQ,Lubetzky_Sly,univ_ising_cutoff}. The key point is estimating the variance by controlling the two-point correlation function. 
\begin{Proposition}
    \label{prop: covariance estimate general}
    Assuming either (L1) or (L2). 
   There exists $c=c(\kappa , \bar{b}, c_+,c_-)>0$ such that for any $x \in \T_N$ and  for any $t>0$
    \begin{align*}
       \sum_{y \in \T_N}\left|\operatorname{Cov}(\bot_t(x),\bot_t(y))\right| \leq c
    \end{align*} 
    The same holds for $\top$.
\end{Proposition}

\begin{proof}[Proof of Theorem \ref{thm: Lower bound general}]

    We have for any $t>0$, 
    $$
    \bbE\left({\sum_{x=1}^N \bot_t(x)}\right)= N \alpha_t^N(\bot)  \qquad \text{and} \qquad \bbE\left({\sum_{x=1}^N \top_t(x)}\right)= N \alpha_t^N(\top)\,.
    $$
    By attractivity of the process, we have
     $$
    \bbE\left({\sum_{x=1}^N \bot_t(x)}\right) \leq \bbE_{\iota_{\rm eq}}\left({\sum_{x=1}^N \eta(x)}\right)\leq  \bbE\left({\sum_{x=1}^N \top_t(x)}\right) \,.
    $$ 
    Therefore, 
    $$
    \max\left\{\bbE\left({\sum_{x=1}^N \bot_t(x)}\right) - \bbE_{\iota_{\rm eq}}\left({\sum_{x=1}^N \eta(x)}\right),   \bbE\left({\sum_{x=1}^N \top_t(x)} \right)  - \bbE_{\iota_{\rm eq}}\left({\sum_{x=1}^N \eta(x)}\right) \right\} \leq \Esp{\mathpzc{r}_t}\,.
    $$
    On the other hand, Proposition \ref{prop: covariance estimate general} implies:
    \begin{equation*}
    \operatorname{Var}\left[\sum_{x=1}^N \bot_t(x) \right] \leq C N, \quad \operatorname{Var}\left[\sum_{x=1}^N \top_t(x) \right] \leq C N \quad \text{and} \quad \operatorname{Var}_{\iota_{\rm eq}}\left[\sum_{x=1}^N \eta(x) \right]\leq C N \,.
   \end{equation*}

   Therefore, we can deduce by the second moment method \cite[Proposition 7.9]{LP}:
   \begin{align*}
    \mathrm{d}_{\mathrm{TV}}\left(\mathbb{P}^ \bot_t, \iota_{\rm eq}\right) 
    &\geq 1-8 \frac{\max \left\{\operatorname{Var}_{\iota_{\rm eq}}\left[\sum_{x=1}^N \eta(x) \right] , \operatorname{Var}\left[\sum_{x=1}^N \bot_t(x) \right]\right\}}{\Esp{\mathpzc{r}_t}^2} \, ,\\
    & \geq 1- C\frac{N}{\Esp{\mathpzc{r}_t}^2}\, .
   \end{align*}
Similarly, we obtain the same bound for $\top$.  For any $\theta \in(0,1)$ and $t\leq T^N_{\rm eq}(1) - \frac{1}{2c_1} \log \left(\frac{C}{1-\theta}\right)$, we obtain thanks to Lemma \ref{lemma: control of REDs exp control} 
\begin{align*}
    \max\left[\mathrm{d}_{\mathrm{TV}}\left(\mathbb{P}^ \bot_t, \iota_{\rm eq}\right) ,\mathrm{d}_{\mathrm{TV}}\left(\mathbb{P}^ \top_t, \iota_{\rm eq}\right) \right]
    &\geq \theta\,.
   \end{align*} 
   Thus,  \eqref{cutoff of the predictive time} implies Theorem \ref{thm: Lower bound general}.
   
\end{proof}

We end this subsection by outlining the main ideas behind the proof strategy for Proposition \ref{prop: covariance estimate general}. Fix two distinct points $x \neq y$. Under assumption \eqref{cond: trivial cluster size} or \eqref{cond: condition general cluster size}, we expect that the correlation created between $\eta(x)$ and $\eta(y)$ remains limited. In other words, we aim to show that they belong to two separate information clusters. Consequently, it is crucial to control either
\begin{itemize}
\item Both the update history of $x$ and $y$. Then, we need to guarantee with high probability that there is no first-class datum in common. 
\item The size of the whole cluster information. 
\end{itemize} 
To this end, it will be convenient to introduce an auxiliary idealised branching processes that respectively either
\begin{itemize}
    \item explores the history while ignoring any correlations backwards in time, or
    \item explores the cluster in both directions, backwards and forward in time.
\end{itemize}
Constructing such processes is the purpose of subsections \ref{subsec: idealised branching process} and \ref{subsec: idealised branching process2}.   In Subsections \ref{subsection: Density profile and correlation control} and \ref{subsec Two point correlation control}, we establish Proposition \ref{prop: covariance estimate general} assuming respectively \eqref{cond: trivial cluster size} or \eqref{cond: condition general cluster size}. 

\subsection{Upper bound} 
In this section, we prove Theorem \ref{thm: Upper bound}. Let us start by observing that for any $\eta_0,\eta_0^\prime \in \{0,1\}^{\T_N}$
\begin{equation*}
    \VT{P_t^{\eta_0}- P_t^{\eta_0^\prime}} = \VT{Q_t^{\xi_
    0}\circ \iota^{-1}- Q_t^{\eta_0^\prime}\circ \iota^{-1}} \leq \VT{Q_t^{\xi_
    0}- Q_t^{\xi_0^\prime}}
\end{equation*}
where $Q_t^{\xi}$ is the law of the expanded process started from a configuration $\xi$ and where $$ \forall x \in \T_N, \,\, \forall i \in[\kappa]\,,\quad \xi_0(x, i)= \ind_{i \leq \eta(x)}\quad \text{and} \quad \xi_0^\prime(x, i)= \ind_{i \leq \eta^\prime(x)}\,.$$
Therefore, it is sufficient to study the mixing of the expanded process, relying on the ideas of Tran \cite[Theorem 1]{HQ}. Given $\eps>0$, let us now fix, 
$$
\ft = T^N_{\rm eq}(1)  + \eps \log(N)  \quad ; \quad t= 1\, .
$$
Let $c_3,c_4$ be the constants in Lemma \ref{lemma: there is a minimum of well positionned Zs} for $a=4$, and we choose $D_N= c_3 \log(N)$.\\

First, we construct the BEP with $E= V_{N,\kappa}$ until time $\ft>0$. Conditionnally on $\cF_\ft^h$ (defined in Subsection \ref{subsection: RGB}), we consider two expanded process $(\zeta_s)_{s \geq 0}$ (resp. $(\zeta^\prime_s)_{s \geq 0}$) started at $(\zeta_{u, \ft}(\xi_0))$ (resp. $((\zeta_{u, \ft}(\xi^\prime_0))$).  Therefore, by the  Markov property of the expanded process and Lemma \ref{lemma: link between bep and expanded}, we have for any $s\geq 0$
$$
\zeta_s \, \, \overset{\cL}{=} \, \, \xi_{\ft + s} \quad \text{and} \quad \zeta^\prime_s \, \, \overset{\cL}{=} \, \, \xi^\prime_{\ft + s}\,,$$
where $(\xi_{s})_{s \ge 0}$ (resp. $(\xi^\prime_{s})_{s \ge 0}$) is an expanded process started at $\xi_0$ (resp. $\xi^\prime_0)$. We have
\begin{align*}
    \VT{Q_{t+\ft}^{\xi_ 0}- Q_{t+\ft}^{\xi_0^\prime}} &= \max_A\left|Q_{t+\ft}^{\xi_ 0}(A)- Q_{t+\ft}^{\xi_0^\prime}(A)\right|\\
    &= \max_A \bbE\left[Q_{t}^{(\zeta_{u, \ft}(\xi_0))}(A)- Q_{t}^{(\zeta_{u, \ft}(\xi^\prime_0))}(A)\right]\\
    &\leq \bbE\left[ \VT{Q_{t}^{(\zeta_{u, \ft}(\xi_0))}- Q_{t}^{(\zeta_{u, \ft}(\xi^\prime_0))}}\right]
 \end{align*}
By a direct induction and Lemmas \ref{lemma: RGB conditionnally on the sigma algebra}, \ref{lemma: one red control} and \ref{lemma: RHS- one red control}, we obtain\footnote{We used the notation $\lesssim$ to hide universal constants.}
 \begin{align*}
    \VT{Q_{t+\ft}^{\xi_ 0}- Q_{t+\ft}^{\xi_0^\prime}} &\lesssim 2  \bbE\left[ \mathpzc r _ \ft \left(\frac{D_N^2}{\sqrt{\beta_N}}+  \bbP_{\sB(\ft)} \left[ \exists s \in [0, t],  
        \max_{ u \in V_{N,\kappa}} \operatorname{dist}(u, \mathpzc B _s) \geq  D_N \right] \right) ^{\frac 1 2}\right]\\
    &\lesssim  \sqrt{\frac{D_N^2}{\sqrt{\beta_N}}} \Esp{\mathpzc r _ \ft } +N \bbE\left[   \bbP_{\sB(\ft)} \left[ \exists s \in [0, t],  
        \max_{ u \in V_{N,\kappa}} \operatorname{dist}(u, \mathpzc B _s) \geq  D_N \right] \right]^{\frac 1 2}
 \end{align*}
where we used Cauchy-Schwarz in the last inequality.  \\

By using the homogeneity of Poisson point processes, we observe that the law of the process $\cB$,  $(\mathpzc B_s)_{s \in [0,t]}$ started with $\mathpzc B_0 = \sB(\ft)$ has the same law as $(\sB^\prime(\fs))_{\fs \in [\ft, \ft +t]}$ resulting from the construction of a BEP until $\ft +t$. Therefore, we have 
\begin{align*}
    \bbE\left[   \bbP_{\sB(\ft)} \left[ \exists s \in [0, t],  
        \max_{ u \in V_{N,\kappa}} \operatorname{dist}(u, \mathpzc B _s) \geq  D_N \right] \right] =     \bbP\left[ \exists \fs \in [\ft, \ft + t],  
        \max_{ u \in V_{N,\kappa}} \operatorname{dist}(u, \sB^\prime(\fs)) \geq  D_N \right]
\end{align*}
which implies by Lemma \ref{lemma: there is a minimum of well positionned Zs}
\begin{align*}
    \VT{P_{t+\ft}^{\eta_0}- P_{t+\ft}^{\eta_0^\prime}}  &\leq  \VT{Q_{t+\ft}^{\xi_ 0}- Q_{t+\ft}^{\xi_0^\prime}}\lesssim \sqrt{\frac{D_N^2}{\sqrt{\beta_N}}} \Esp{\mathpzc r _ \ft } +N \sqrt{\frac{1}{N^4}}
 \end{align*}
 Finally, we have by Lemma \ref{lemma: control of REDs exp control} 
 $$
  \Esp{\mathpzc r _ \ft }  \leq {}{\sqrt{N}} N^{-c_2 \varepsilon }\,.
 $$
The assumption $\beta_N=N^2$  ensures that $ \VT{P_{t+\ft}^{\eta_0}- P_{t+\ft}^{\eta_0^\prime}}$ converge to zero and implies therefore the result. \qed

\section{Density profile and correlation control}
\label{section: Density profile and correlation control}

{In this section, we derive quantitative estimates on the update history and the information clusters, which provide the key ingredients for the analysis of the density profile and the two-point correlation function.

We first introduce in Subsection \ref{subsec: idealised branching process} an idealised branching process, called the \emph{History Exploration Process} (HEP), to control the size of the update history. This follows the general strategy of \cite{de_masi_ferrari_lebowitz} and \cite{HQ}. Nevertheless, the Misanthrope dynamics requires an adaptation of the exploration procedure. Once the adapted exploration process is defined, the estimates on the density profile and the update history are derived in Subsections \ref{subsection: Density profile and correlation control} and \ref{subsec: Coupling the HEP with BEP}.

We then introduce in Subsection \ref{subsec: idealised branching process2} a second idealised branching process, called the \emph{Cluster Exploration Process} (CEP), to control the information clusters relying on the information percolation framework of \cite{univ_ising_cutoff}. Our construction differs from that of Lubetzky and Sly by exploring the entire cluster that it explores all branches of the cluster simultaneously through a Markovian dynamic, rather than revealing it one branch at a time.  The size of the cluster is then controlled through a birth-and-death argument.}

\subsection{A first idealised branching process}
\label{subsec: idealised branching process}

Fix  $E $ be a subset of $\cW$ (defined in \eqref{eq-def: the set of words}) where \footnote{The condition \eqref{eq: condition injectivity of labels} ensures the injectivity of the labels.}
\begin{equation}
    \label{eq: condition injectivity of labels}
\forall w, w^\prime \in \cW,\qquad w\nsucc w^\prime.
\end{equation}
Let us introduce an analogue construction to the BEP that no longer evolves on $V_{N,\kappa}$:
\begin{enumerate}
    \item At time $0$, for any $w \in E$, we put a datum labeled as $w$. Each datum has independent Exponential clocks as follows:
    \begin{itemize}
        \item {\bf Glauber clocks}: $q_-$ independent exponential clocks of intensities $\lambda_{-1},.., \lambda_{-q_-}$ and $q_+$ independent exponential clocks of intensities $\lambda_{1},..,\lambda_{q_+}$. 
        \item {\bf Refresh clocks}: One exponential clock  of rate $\upsilon_-$ and we associate to it a random variable $\mathfrak{c} \sim \cB(\bar{\rho}).$  
        \item {\bf Correlated Left exchange clocks and Correlated Right Exchange clocks}: $2\kappa$  exponential clock on of rate $\bar b$ and we associate to it a random variable $v \sim \cU([0,1])$.  
        For convenience, we label them $(\pm,j)$ with $j\in[\kappa]$ and this is associated with exchanging with the left/right jumps and the $j$th box.  
    \end{itemize}
    \item For any $w$, each time one of these clocks rings, the following then occurs:
    \begin{itemize}
        \item If it is a Refresh mark, the datum $w$ is killed. 
        \item If it is a Glauber clock, we kill $w$, and we create $2\kappa(2m+1)$ new data on all the sites labelled $wu$  where $u \in B(0,m)$.
        \item If it is a correlated left jump marked with $v$, we kill $w$  and we create new data $2\kappa$ children labeled $wu$ with $u\in \{-1,0\}\times [\kappa]$. Same for the right jumps.
    \end{itemize}
\end{enumerate}

The History Exploration process (HEP) started from $E$ refers to the set of all marks of all the data that lived between $0$ and $\ft$. We will denote it by $(\widetilde{ \fC}_{\ft}(E))_{\ft \geq 0}$. 

Analogously, we say that the datum $w$ is alive at time $\ft$ if it was born before time $\ft$ and has not been killed by time $\ft$. For any $\fs \leq \ft$, we define $\widetilde {\fD}_{\fs}(E)$ as the set of alive data at time $\fs$. {Once again, we omit the dependence in $E$ in the notation when it is clear from the context.}\\

\paragraph{\emph{Reveal of the content.}}
Let us fix $\ft>0$. Conditionally on $\widetilde {\fC}_{\ft}$, given $\widetilde{\sigma} \in \{0,1\}^{\widetilde {\fD}_{\ft}}$ the content of each alive data in $\widetilde {\fD}_{\ft}$, we can deterministically reveal as before the content of all data that lived between $0$ and $\ft$ according to the same rules as the BEP process. Therefore, conditionnally on the IBP, for any data $w$ alive at some point in the time interval $[0,\ft]$, there exists an \emph{non-decreasing} function $\widetilde F_{w,\ft}$ on $\{0,1\}^{\widetilde{\fD}_{\ft}}$ such that for any spin configuration $\widetilde{\sigma}$ on $\widetilde{\fD}_{\ft}$ satisfies
$$
\widetilde{\sigma}_{ w }=\widetilde{F}_{ w , {\ft}}\left(\left(\widetilde{\sigma}_{ w ^{\prime}}\right)_{{w}^{\prime} \in \fW_{\ft}}\right) .
$$

Let us emphasise that $(\widetilde{\fD}, (\widetilde{F}_u)_{u\in E})$ is a Markov process on 
$$
\bigcup_{\substack{\fW \subset \cW\\ \operatorname{Card}(\fW)< \infty}} \{\fW\}\times {(\{0,1\}^E)}^{\{0,1\}^\fW}
$$
with a compatible generator.

\paragraph{\emph{The idealised update history:}} We also define the update history for any $A \subset  \bigcup_{\fs\in [0,\ft]} \widetilde{\fD}_\fs$ as:
\begin{equation}
    \label{eq-def: idealised history update}
    \widetilde{\mathscr{H}}_{A}(E,\ft)= \left\{ w^\prime \in \bigcup_{\fs\in [0,\ft]} \widetilde{\fD}_\fs; \text{ such that there exist \,} w \in A, \,  w^\prime \succ w \right\} \,.
\end{equation}
and we denote by $\widetilde{\fh}_{A}(E,\ft)$ its cardinal. We once again omit the dependence in $E$, when it is clear from the context.

\begin{Lemma}[Density profile for HEP]
    \label{lemma: Density profile for HEP}
    For any $\ft>0$, $o\in \{0,1\}$  and $uV_{N,\kappa}$, we have:
    \begin{equation}
        \alpha_o(\ft)= \kappa \bbE^{\rm HEP}_{\{u\}}\left[\widetilde{F}_{u, \ft}((o)_{w^\prime \in \fW_{\ft} })\right]
    \end{equation}
    where  $\alpha_0$ (resp. $\alpha_1$) is the solution of equation \eqref{Eq:reactdiff} initialised at $0$ (resp. $\kappa$).
\end{Lemma}

\begin{proof}
    Given  $o \in \{0,1\}$ and $u=(x,i)$. Let us denote by:
            \begin{equation}
            \rho_o(u,\ft)=  \bbE^{\rm HEP}_{\{u\}}\left[\widetilde{F}_{u, \ft}((o)_{w^\prime \in \fW_{\ft} })\right]
        \end{equation}
    Notice that it the law of $\widetilde{F}_{u, \ft}((o)_{w^\prime \in \fW_{\ft} })$ does not depend on the initial label $u$. Therefore, $\rho_o(u,\ft)$ is constant in $u$, and we omit $u$ in the expression from now on.
    
    If we start the HEP at $E_u = \{uv\}_ {v \in B(0,m)}$ for a given $y\in \T_N$, e observe that $(\widetilde{\sigma}_{uv})_{v \in B(0,m)}$ are i.i.d and follow a Bernoulli distribution with parameter $\rho_o(\ft)$.
    For convenience, we define $$\iota_{u(k,\cdot)}=\sum_{j=1}^\kappa\widetilde{\sigma}_{u\cdot(y,j)}\,,$$
    for $y \in \llbracket -m,m \rrbracket$.
    We also recall that if $(Z_i)_{i=1}^\kappa \sim \otimes_{j=1}^\kappa \cB(\rho)$ then
    \begin{equation}
        \label{eq: conditionnal bernoulli-sum}
         \sum_{j=1}^\kappa Z_j \sim \cB(\kappa,\rho) \qquad \text{and} \qquad \Esp{Z_1 \left| \sum_{j=1}^\kappa Z_i \right.} =\frac{\sum_{j=1}^\kappa Z_l }{\kappa}\,.
    \end{equation}
    
    Thanks to the Backward Kolmogorov equation for the Markovian process  $(\widetilde{\fD}, (\widetilde{F}_u)_{u\in E})$, we obtain:
    $$
    \partt  \rho_o(\ft) =    M_{L, \ft} +    M_{R, \ft} +R_{\ft} + G_{\ft}^- +G_{\ft}^+ \,,
    $$
    where:
    \begin{itemize}
        \item $M_L$ corresponds to the correlated  left exchange terms:
        \begin{align*}
        M_{L,\ft} &= \sum_{j=1}^\kappa \bbE^{\rm HEP}_{E_u}\left[\tilde{b}\left(\iota_{u(-1,\cdot)},\iota_{u(0,\cdot)}\right) \widetilde{\sigma}_{u(-1,j)} \right]+\bbE^{\rm HEP}_{E_u}\left[ (1-\tilde{b}\left(\iota_{u(-1,\cdot)},\iota_{u(0,\cdot)}\right)\widetilde{\sigma}_{u(-1,j)} ) \widetilde{\sigma}_{u(0,i)}\right] - \rho_0(\ft) \\
        & + \bbE^{\rm HEP}_{E_u}\left[\tilde{b}\left(\iota_{u(0,\cdot)},\iota_{u(1,\cdot)}\right) \widetilde{\sigma}_{u(1,j)}\widetilde{\sigma}_{u(0,i)} \right]+ \bbE^{\rm HEP}_{E_u}\left[ (1-\tilde{b}\left(\iota_{u(0,\cdot)},\iota_{u(1,\cdot)}\widetilde{\sigma}_{u(0,i)}\right)) \widetilde{\sigma}_{u(0,i)}\right] - \rho_0(\ft)\\
        & = \sum_{j=1}^\kappa  \bbE^{\rm HEP}_{E_u}\left[\tilde{b}\left(\iota_{u(-1,\cdot)},\iota_{u(0,\cdot)}\right) \frac{\iota_{u(-1,j)}}{\kappa} \right]
        -\bbE^{\rm HEP}_{E_u}\left[ \tilde{b}\left(\iota_{u(-1,\cdot)},\iota_{u(0,\cdot)}\right)\frac{\iota_{u(-1,j)}}{\kappa} \frac{\iota_{u(0,i)}}{\kappa}\right] \\
        & +\bbE^{\rm HEP}_{E_u}\left[\tilde{b}\left(\iota_{u(0,\cdot)},\iota_{u(1,\cdot)}\right) \frac{\iota_{u(1,j)}}{\kappa}\frac{\iota_{u(0,i)}}{\kappa}\right]- \bbE^{\rm HEP}_{E_u}\left[ \tilde{b}\left(\iota_{u(0,\cdot)},\iota_{u(1,\cdot)}\right)\frac{\iota_{u(0,i)}}{\kappa} \right]\\
        &  =0\,.
        \end{align*} 
        Indeed, after the left exchange directed towards the datum $u(0, i)$, its content is one either because the jump succeeded and $\widetilde{\sigma}_{u(-1,j)}$ was equal to one or the jump failed and $\widetilde{\sigma}_{u(0, i)}$ was equal to one. Similarly, for the left exchange out of the datum $u(0, i)$, its content is one either because even though the jump succeeded and $\widetilde{\sigma}_{u(1,j)}$ was equal to one or the jump failed and $\widetilde{\sigma}_{u(0, i)}$ was equal to one. In the second equality, we used \eqref{eq: conditionnal bernoulli-sum}.
        \item $M_R$ corresponds to the correlated right exchange marks, and an analogous computation shows that it is equal to zero.
        \item $R$ corresponds to the refresh terms:
        \begin{align*}
            R_t= \upsilon_- (\bar{\rho} - {\rho_o(\ft)}) \,.
        \end{align*}
        \item $G^-$ corresponds to the Glauber clocks that may kill a datum:
        \begin{align*}
            G_t^-&= \sum_{j=-q_-}^{-1} \lambda_j \left\{  \bbE^{\rm HEP}_{E_u}\left[(1-f_j)\left(\left(\iota_{u(k,\cdot)}\right)_{k \in B(0,m)}\right) \widetilde{\sigma}_{u(0,i)}\right] - {\rho_o(\ft)} \right\} \\
            &= - \sum_{j=-q_-}^{-1} \lambda_j  \bbE^{\rm HEP}_{E_u}\left[f_j\left(\left(\iota_{u(k,\cdot)}\right)_{k \in B(0,m)}\right) \frac{\iota_{u(0,\cdot)}}{\kappa}\right]\,.
        \end{align*}
        Indeed, after a Glauber mark for $j \in \llbracket -q_-,-1\rrbracket $, the content is equal to one only if it failed and the content $\widetilde{\sigma}_{u(0, i)}$ was already equal to one.
        \item Similarly, $G^+$ corresponds to the Glauber clocks that may add a datum:
        \begin{align*}
            G_t^+
            &= \sum_{j=1}^{q_ +} \lambda_j \left\{  \bbE^{\rm HEP}_{E_u}\left[f_j\left(\left(\iota_{u(k,\cdot)}\right)_{k \in B(0,m)}\right)  + \left(1-f_j\left(\left(\iota_{u(k,\cdot)}\right)_{k \in B(0,m)}\right)\right) \widetilde{\sigma}_{u(0,i)} \right] - {\rho_o(\ft)} \right\} \\
            &=\sum_{j=1}^{q_ +} \lambda_j  \bbE^{\rm HEP}_{E_u}\left[  f_j\left(\left(\iota_{u(k,\cdot)}\right)_{k \in B(0,m)}\right) \frac{\kappa-\iota_{u(0,\cdot)}}{\kappa}\right] \,.
        \end{align*}
    \end{itemize}
   Therefore, we obtain by \eqref{eq: decomposition of c}
$$
\kappa \partt \rho_o(\ft)= R(\kappa \rho_o(\ft))\,,
$$
where $R$ is defined in \eqref{eq: reaction term}.
\end{proof}

Let us give a first control on the history set.
\begin{Lemma}[Control of the HEP]
    \label{lemma: trivial cluster size}
    Assuming \eqref{cond: trivial cluster size},
     there exists $\mu= \mu(\kappa , \bar{b}, c_+,c_-)>0$ and $\gamma=\gamma(\kappa , \bar{b}, c_+,c_-)>0$, such that for any $E \subset V_{N,\kappa}$ and $\ft>0$:
$$
\mathbb{E}_E^{\rm HEP}\left[e ^{\lambda \widetilde{\fh}_{E}(E,\ft)}\right]< \exp(\gamma|E|)\,.
$$
\end{Lemma}

\begin{proof} We denote by $\widetilde \fd_\ft$ the cardinal of $\widetilde{\fD}_{\ft}$.
We observe that $\left(\widetilde{\fd}_{\ft}\right)_ {t \geq 0}$ is a classical Birth and Death process: 
\begin{itemize}
    \item with rate $\upsilon_- \widetilde{\fd}_{\ft} $, it decreases by $1$
    \item with rate $\upsilon_+ \widetilde{\fd}_{\ft} $, $\widetilde{\fd}_{\ft} $ increases by $ \kappa (2m+1) -1 $ and $\widetilde{\fh}_{E}(\ft)$ increases by  $ \kappa (2m+1) $,
    \item  with rate $ 2\bar{b}\kappa\widetilde{\fd}_{\ft} $, $\widetilde{\fd}_{\ft} $  increases by $2 \kappa -1$ and $\widetilde{\fh}_{E}(\ft)$ increases by  $2 \kappa $.
\end{itemize} 
Therefore, we obtain by the Kolmogorov forward equation for any $\lambda, \mu>0$
\begin{align*}
    &\partt \mathbb{E}\left[e ^{\mu \widetilde{\fh}_{E}(\ft)+\lambda \widetilde{\fd}_{\ft}}\right] \\
    &=   \bbE\left[ \upsilon_+  \widetilde{\fd}_{\ft} \left[ e^{\mu \widetilde{\fh}_{E}(\ft)+\lambda \widetilde{\fd}_{\ft} + \kappa (2m+1)(\mu+\lambda) - \lambda}  - e^{\mu \widetilde{\fh}_{E}(\ft)+\lambda \widetilde{\fd}_{\ft}}\right]\right]\\
    & +\bbE\left[2\bar b \kappa
   \widetilde{\fd}_{\ft} \left[ e^{\mu \widetilde{\fh}_{E}(\ft)+\lambda \widetilde{\fd}_{\ft} +2\kappa (\mu+\lambda) - \lambda}  - e^{\mu \widetilde{\fh}_{E}(\ft)+\lambda \widetilde{\fd}_{\ft}}\right]\right]  \\
    &+\bbE\left[  \upsilon_- \widetilde{\fd}_{\ft}\left[ e^{\mu \widetilde{\fh}_{E}(\ft)+\lambda \widetilde{\fd}_{\ft} - \lambda}  - e^{\mu \widetilde{\fh}_{E}(\ft)+\lambda \widetilde{\fd}_{\ft}}\right]\right] \\
    &=  \bbE\left[ \widetilde{\fd}_{\ft} e^{\lambda \widetilde{\fh}_{E}(\ft)+2\kappa \lambda}
     \left[\upsilon_+ (e^{\kappa (2m+1)(\mu+\lambda)  - \lambda}  - 1) +  2\bar b \kappa (e^{2\kappa (\mu+\lambda)  -\lambda}  - 1)  + \upsilon_- (e^{- \lambda}  - 1) \right]\right] 
\end{align*}
 If $\lambda$ and $\mu$ are sufficiently small,  \eqref{cond: trivial cluster size}implies that $\partt \mathbb{E}\left[e ^{\mu \widetilde{\fh}_{E}(\ft)+\lambda \widetilde{\fd}_{\ft}}\right] \leq 0 $. Consequently, $ \mathbb{E}\left[e ^{\mu \widetilde{\fh}_{E}(\ft)}\right] \leq e^{(\lambda+ \mu)|E|} $. 
\end{proof}

\subsection{Coupling the HEP with BEP}
\label{subsec: Coupling the HEP with BEP}
In this subsection, we assume that $\beta_N=N^2$. Let us start by introducing a coupling between BEP and HEP, starting from $E$, a subset of $V_{N,\kappa} $, by identifying data with the same label as long as there is no second-class data. More precisely, for any living data $w$ at position $u=(x,i)$:
\begin{enumerate}
    \item We identify both BEP and HEP Glauber and Refresh clocks. 
    \item Regarding the exchange clocks, we identify the BEP-correlated exchange clocks associated with edge $\{(x\pm 1,j),(x, i)\} $ with the HEP's correlated exchange associated with data $w$ labelled $(\pm,j)$, for any $j\in [\kappa]$.
    \item Each time $w$ encounters an exclusion mark, the correlated exchange clocks are updated to respect the previous choice. 
\end{enumerate}
The coupling fails before $t$ if there exists second-class data born before time $t$. Thanks to assumption \eqref{cond: trivial cluster size}, the coupling is indeed successful with high probability, as we shall prove it below. 
\begin{Lemma}   
    \label{lemma: success of the coupling}
    Given $E \subset V_{N,\kappa}$ and assuming \eqref{cond: trivial cluster size}, there exists a constant $C>0$, such that for any $\ft\geq 0$,
    $$
    \bbP_E\left[ \exists \fs \leq \ft, \qquad \fD_s \neq \fW_s \right] \leq \frac{C}{N} \,,
    $$
    where $\bbP_E$ is the probability associated to the coupling started at $E$.
\end{Lemma}

\begin{proof}[Proof of Lemma \ref{lemma: success of the coupling} ]
    We work on the event where the coupling fails
    $$\mathscr{F}_{\ft} =\left\{ \exists {\fs} \leq \ft, \qquad \fD_{\fs} \neq \fW_{\fs} \right\}\,. $$
    We introduce the first time a second-class particle appears by 
    $$\tau= \inf\left\{ \fs \leq \ft, \, \text{such that, }\qquad \fD_{\fs} \neq \fW_{\fs} \right\} < \ft \,.$$ 
    By construction of the processes, $\tau$ is the first death time of some first-class datum, either due to correlated exchange or Glauber mark.  At this death, there is a creation of a second-class datum due to the presence of another first-class datum in the $m$ neighbourhood.\\
    We will denote by $w^{\star \star}$ and $ w^\star$ the two first-class data involved.  Without loss of generality, $w^{\star\star}$ is the firstborn of the two data. \\
    
    We denote $T^b_\star$ (resp. $T^b_{\star \star}$) the birth time of these data. Similarly, we denote by $T^d_\star$ (resp. $T^d_{\star \star}$) the death time of these data. From this discussion, we can write:
    \begin{align*}
        \Prob{\mathscr{F}_t} 
        &\leq \bbE_E\left[\sum_{w^\star \in \mathscr{H}_{\tau ^-}} \sum_{w^{\star \star}\in \fW_{{T^{b}_\star}}}  \mathds{1}_{\tau=\min[{T^d_\star ,T^d_{\star\star} }]\leq t} \mathds{1}_{\operatorname{dist}\left(\ell_{\min[{T^d_\star ,T^d_{\star\star} }]}(w^{\star}), \ell_{\min[{T^d_\star ,T^d_{\star\star} }]}(w^{\star \star })\right)\leq m} \right]\\
        &\leq \sum_{w^\star \in\cW} \bbE_E\left[ \mathds{1}_{{T^{b}_\star}\leq \tau}\mathds{1}_{{T^{b}_\star} \leq \ft} \sum_{w^{\star \star}\in \fW_{{T^{b}_\star}}} \mathds{1}_{\operatorname{dist}\left(\ell_{\min[{T^d_\star ,T^d_{\star\star} }]}(w^{\star}), \ell_{\min[{T^d_\star ,T^d_{\star\star} }]}(w^{\star \star })\right)\leq m} \right]\,.
    \end{align*}
    By taking the conditionnal expectation with respect to $\cF_{T^b_\star}$, we have:
    \begin{align*}
        \Prob{\mathscr{F}_t} \leq \sum_{w^\star \in\cW} \bbE_E\left[ \mathds{1}_{{T^{b}_\star}\leq \tau}\mathds{1}_{{T^{b}_\star} \leq t} \sum_{w^{\star \star}\in \fW_{{T^{b}_\star}}} \bbE_E\left[\left. \mathds{1}_{\operatorname{dist}\left(\ell_{\min[{T^d_\star ,T^d_{\star\star} }]}(w^{\star}), \ell_{\min[{T^d_\star ,T^d_{\star\star} }]}(w^{\star \star })\right)\leq m}\right|\cF_{T^b_\star}\right] \right]\,.
    \end{align*}

    Notice that $\min[{T^d_\star ,T^d_{\star\star} }]-T^b_\star $ is an Exponential random variable with parameter $2 \bar b + 2 \upsilon_+$.

    Moreover, conditionnally on $\cF_{T^b_\star}$, $\ell(w^{\star})$ and $\ell(w^{\star \star})$ are the occupied positions of a two-particle exclusion process on the expanded graph $\bbG_{N,\kappa}$ until $\min[{T^d_\star ,T^d_{\star\star} }]$. Thanks to Lemma \ref{antioncentration lemma}, there exist a universal constant $c>0$ such that 
    \begin{equation*}
        \Prob{{\operatorname{dist}\left(\ell_{\min[{T^d_\star ,T^d_{\star\star} }]}(w^{\star}), \ell_{\min[{T^d_\star ,T^d_{\star\star} }]}(w^{\star \star })\right)\leq m}\rvert\cF_{T^b_\star} } \leq \frac{c}{N}\,, 
    \end{equation*}
    where we used $\beta_N= N^2$.
    Therefore, we obtain:
    \begin{align*}
        \Prob{\mathscr{F}_t} 
        &\leq \sum_{w^\star \in\cW} \bbE_E\left[ \mathds{1}_{{T^{b}_\star}\leq \tau}\mathds{1}_{{T^{b}_\star} \leq t} \sum_{w^{\star \star}\in \fW_{{T^{b}_\star}}} \frac{c}{N} \right]\\
        &\leq \frac{c}{N} \bbE\left[ |\mathscr{H}_{\min(\tau^-,t)}|^2 \right] \leq \frac{c}{N} \bbE\left[ |\widetilde{\mathscr{H}}_{\min(\tau^-,t)}|^2 \right] \leq \frac{c}{N} \bbE\left[ |\widetilde{\mathscr{H}}_{t}|^2 \right] 
    \end{align*}
    where we used that $\mathscr{H}_{\min(\tau^-,t)}=  \widetilde{\mathscr{H}}_{\min(\tau^-,t)}$. The result follows from Lemma \ref{lemma: trivial cluster size}.
\end{proof}
\begin{remark}
     The anti-concentration argument (Lemma \ref{antioncentration lemma}) is used because the exact locations of the data are unknown, and we need to control their positions to avoid the creation of correlation. The diffusive scaling $\beta_N= N^2$  ensures that the positions of the two data  are ``close to equilibrium ''.\\
     We also emphasise that this condition is required for this method to estimate the discrete profile. In fact, this coupling yields a replacement lemma in which we compare the local dynamics with an equilibrium configuration of the partial exclusion process.\\
    Consequently, it is clear that the exclusion component must be of the correct order for it to work. If one is only interested in bounding correlations, it is sufficient to control the size of the entire information cluster (see Definition \ref{def: information cluster}). {This leads to the statement under assumption \eqref{cond: condition general cluster size} given in Theorem \ref{thm: Lower bound general}.}
\end{remark}
\subsection{Density profile and correlation control}
\label{subsection: Density profile and correlation control} 
In this subsection, we still assume $\beta_N=N^2$, and we aim to prove Proposition \ref{prop: covariance estimate general} and Proposition \ref{prop: expression of predictive time} under assumption \eqref{cond: trivial cluster size} . Let us first recall the expression of the reaction term:
 \begin{equation*}
    R (\alpha) = \mathbb{E}_{\cB\left(\kappa,\frac{\alpha}{\kappa}\right)}\left[ (\kappa-\eta(0)) c_+(\eta) - \eta(0) c_-(\eta)\right].
 \end{equation*}
where $R$ is defined in \eqref{eq: reaction term}.
For $o \in\{0,1\}$,For $o \in \{0,1\}$, we denote by $(\alpha_o(t))_{t\ge0}$ the solution started from $\alpha_o(0)=o\kappa$ to 
 \begin{equation}
    \partt \alpha_o= R(\alpha_o)
 \end{equation}
Under the assumption \eqref{cond: trivial cluster size}, the following lemma holds. 
\begin{Lemma}
    \label{lemma: technical lemma on R}
    There exists $\alpha_\star \in[0,\kappa], C,\tilde C>0$ such that for $o\in \{0,1\}$ and $t \geq 0$:
    $$C\ee^{R^\prime(\alpha_{\star}) t} \leq |\alpha_o(t)- \alpha_\star| \leq \tilde C\ee^{R^\prime(\alpha_{\star}) t}\,,$$
    {where $R^\prime(\alpha_{\star}) <0$.} 
\end{Lemma}
The proof of this technical lemma is postponed to Appendix \ref{section-appendix: control of Red}. We are now in shape to prove Proposition \ref{prop: expression of predictive time}.
\begin{proof}[Proof of Proposition \ref{prop: expression of predictive time}]
    For any $x \in \T_N$,  we choose $E$ as $\{(x,j)\}_{j=1}^\kappa$. We denote by $\mathscr{S}$ the event of success up to time $t >0$. On $\mathscr{S}$, we have 
    $$ \forall u \in E;\qquad \vee_\ft(u) =  {F}_{u, \ft}((0)_{w^\prime \in \fW_{\ft}})=\widetilde{F}_{u, \ft}((0)_{w^\prime \in \fW_{\ft}})=  \widetilde{\sigma}_{u}\, ,$$
    where $\vee_\ft(u) $ is contructed in \eqref{eq: vee and wedge construction}.
    Thanks to Lemma \ref{lemma: link between bep and expanded} and Lemma \ref{lemma: link expanded and misanthrope process}, we obtain
    \begin{align*}
        \alpha_\ft^N(\bot) &= \sum_{j=1}^{\kappa}\Esp{\vee_{(0,j)}(\ft)}\leq \sum_{j=1}^{\kappa}\Esp{\vee_{(0,j)}(\ft) \mathds{1}_{\mathscr{S}}} +\kappa \Prob{\mathscr{S}}
        \\& \leq  \sum_{j=1}^{\kappa}\Esp{\widetilde{\sigma}_{(0,j)}(\ft)} +\kappa \Prob{\mathscr{S}}
        \\&= \alpha_0(\ft)+\kappa \Prob{\mathscr{S}}.
    \end{align*}
    where we used in the last equality Lemma \ref{lemma: Density profile for HEP}.
    By Lemma \ref{lemma: success of the coupling}, we have
    \begin{equation}
        \label{eq: step 1 density control}
        |\alpha_\ft^N(\bot) - \alpha_0(\ft)| \leq \frac{c}{N}.
    \end{equation}
    Hence, by Lemma \ref{lemma: technical lemma on R}
    \begin{equation}
       C\ee^{R^\prime(\alpha_{\star}) \ft} - \frac{c}{N} \le \left|\alpha^N_\ft(\bot)-\alpha_\star \right| \le  \tilde C \ee^{R^\prime(\alpha_{\star}) \ft} + \frac{c}{N}.
    \end{equation}
    Similarly, we prove the same statement for $\wedge$.
\end{proof}

Let us now prove Proposition \ref{prop: covariance estimate general} under assumption \eqref{cond: trivial cluster size}
\begin{proof}[Proof of Proposition \ref{prop: covariance estimate general} ]
    We denote by $\rho_0(t)= \frac{\alpha_0(t)}{\kappa}$.
    Given $x,y \in \T_N$, it is sufficient to prove 
    $$\Esp{(\bot_x(\ft)-\kappa \rho_0(\ft))(\bot_y(\ft)- \kappa \rho_0(\ft))}\leq \frac{c}{N} \,$$
    thanks to \eqref{eq: step 1 density control}.
    We choose $E = \{(x,j)\}_{j=1}^\kappa \cup \{(y,j)\}_{j=1}^\kappa $. We denote by $\mathscr{S}$ the event on which the coupling defined in Lemma \ref{lemma: success of the coupling} succeed up to time $\ft >0$, i.e
  $$  \mathscr{S}_{\ft} =\left\{ \forall {\fs} \leq \ft, \qquad \fD_{\fs} = \fW_{\fs} \right\}.$$
     On the event $\mathscr{S}$, $$ \forall u \in  E;\qquad \vee_\ft({u}) =  {F}_{u, \ft}((0)_{w^\prime \in \fW_{\ft}})=\widetilde{F}_{u, \ft}((0)_{w^\prime \in \fW_{\ft}})=  \widetilde{\sigma}_{u}\,.$$
    Therefore, we obtain thanks to Lemma \ref{lemma: link between bep and expanded} and Lemma \ref{lemma: link expanded and misanthrope process}:
    \begin{align*}
        &\Esp{(\bot_x(\ft)-\kappa \rho_0(\ft))(\bot_y(\ft)- \kappa \rho_0(\ft))} \\
        &\qquad = \sum_{j=1}^\kappa \sum_{i=1}^\kappa \Esp{(\vee_{(x,i)}(\ft)-\rho_0(\ft))(\vee_{(y,j)}(\ft)-\rho_0(\ft))} \\
        &\qquad  = \sum_{j=1}^\kappa \sum_{i=1}^\kappa \Esp{(\vee_{(x,i)}(\ft)-\rho_0(\ft))(\vee_{(y,j)}(\ft)-\rho_0(\ft)) \mathds{1}_{\mathscr{S}}} +\kappa^2 \Prob{\mathscr{S}}\\
        &\qquad \leq  \sum_{j=1}^\kappa \sum_{i=1}^\kappa \Esp{(\widetilde{\sigma}_{(x,i)}(\ft)-\rho_0(\ft))(\widetilde{\sigma}_{(y,j)}(\ft)-\rho_0(\ft))} +\kappa^2 \Prob{\mathscr{S}}
    \end{align*}
    Since $\widetilde{\sigma}_{(x, i)}$ and $\widetilde{\sigma}_{(y,j)}$ are independent, we have by
    $$
    \Esp{(\widetilde{\sigma}_{(x,i)}(\ft)-\rho_0(\ft))(\widetilde{\sigma}_{(y,j)}(\ft)-\rho_0(\ft))}= \Esp{\widetilde{\sigma}_{(x,i)}(\ft)-\rho_0(\ft)}\Esp{\widetilde{\sigma}_{(y,j)}(\ft)-\rho_0(\ft)}=0 \,
    $$
    where we used Lemma \ref{lemma: Density profile for HEP}. The claim follows from Lemmas \ref{lemma: success of the coupling}.
\end{proof}
\subsection{A second idealised process}
\label{subsec: idealised branching process2}
{In the remainder of this section, we no longer assume $\beta_N=N^2$. In contrast with the update history, the information cluster defined in Definition \ref{def: information cluster} collects all data that may share the same descendants. To estimate their sizes, it will be convenient to introduce a second idealized auxiliary branching process in which the data do not live on the graph $\T_N$, and to explore their clusters in both directions. For each datum $w$ encountered, we associate a direction $d(w) \in \{\pm1\}$ that describes in which time direction it was discovered during the exploration.\\

This idea is based on \cite[Lemma 2.2]{univ_ising_cutoff}. In \cite{univ_ising_cutoff}, the authors explore the information cluster one branch at a time, both backward and forward. Here, we explore it through a Markovian dynamics that discovers the whole cluster. With this aim, we introduce this second idealized process which stochastically dominates the information cluster through an explicit coupling and which we call Cluster Exploration process.}

Let $$\cW^\star= \bigcup_{i=0}^{\infty}\left( V_{N,\kappa} \times [B(0,2 m) \cup \{\# \} ]^i \right)$$
be the set of words starting by an element $u \in V_{N,\kappa}$ and letters $c_1\cdots c_i$ from $[B(0, m) \cup \{\# \} ]^i$. An element $w$ labels the datum that $u$ is connected to due to a series of branching encoded by the  $c_1\cdots c_i$, where $\#$ symbolises a change of time direction.  As usual, if $w^\prime$ is a prefix of $w$, $w$ is referred to as a descendant of $w ^\prime$, and it is denoted by $w\succ w^\prime$.\\

Given  $A$ a subset of $V_{N,\kappa}$, we construct the Cluster Exploration process (CEP) as follows:
\begin{enumerate}
    \item At time $0$, for any $w \in E$, we put a datum labeled as $w$ and $d(w)=-1$. Each datum has independent Poissonian clocks as follows:
    \begin{itemize}
        \item {\bf Glauber clocks}: $(2m \kappa+\kappa) q_-$ independent exponential clocks of rates $\lambda_{-1},.., \lambda_{-q_-}$ and $(2m \kappa+\kappa) q_+$  independent exponential clocks  of rates $\lambda_{1},..,\lambda_{q_+}$. For convenience, we label them $(z,j,q)$ with $q\in \llbracket q_-,q_+ \rrbracket$, $j \in[\kappa]$ and $z \in B(0,m)$.
        \item {\bf Refresh clocks}: One  independent exponential clock of rate $\upsilon_-$ . 
        \item {\bf Correlated Left exchange clocks and Correlated Right Exchange clocks}: $2\kappa$   independent exponential clock of rate ${ \bar{b}} $.  For convenience, we label them $(\pm,j)$ with $j\in[\kappa]$ and this is associated with exchanging with the left/right jumps and the $j$th box.
    \end{itemize}
    \item For any $w$, each time one of these clocks rings, the following then occurs:
    \begin{itemize}
        \item If it is a Refresh clock, the datum $w$ is killed. 
        \item If it is a Glauber clock $(z,j,q)$, we kill $w$  and we create on data on all the sites labeled $wu$  where $u \in B(z,m)$ in the same direction as $d(w)$ (i.e fix $d(wu)=d(w)$). We also create new data on all the sites labeled $w\#u$  where $u \in B(z,m) \setminus \{(0,\kappa)\}$ in the opposite direction as $d(w)$ (i.e fix $d(wu)=-d(w)$). 
        \item If it is a correlated left jump clock, we kill $w$, and we create $2\kappa$ new data labelled $wu$ with $u\in \{-1,0\}$. We  also create $2\kappa-1$ new data labeled $w\#u$ with $u\in \{-1,0\}\setminus \{(0,i)\}$. Same for the right jumps.
    \end{itemize}
\end{enumerate}
We proceed in this way until all data are killed, which corresponds to exploring the whole data cluster. For any $t$, denote by  $\widetilde{ \mathcal{C}}(t)$ the set of all data that has been created before time $t$.
We refer to $\widetilde{ \mathcal{C}}(\infty)$ as the set of all data that has once existed.

Notice that in this construction $t$ does not refer to the time associated with the expanded process but rather the time spent in exploration. \\

\emph{Let us explain the construction:} When we explore the information cluster, we begin by tracing its past. However, it may happen that during this exploration we encounter another datum that belongs to a different update history— one that we have just intersected. This corresponds to the creation of a second-class datum. In this construction, we always assume that this happens; we treat it as a newly discovered cluster, and we therefore explore it as well.

Intersecting a new history may happen for different reasons. Of course, either the site updates due to its own Glauber clock rings or a correlated exchange clock rings. But, one of its neighbours may update first, and this interaction creates a correlation that forces us to branch in the exploration.

\begin{Lemma}
    \label{lemma: information cluster control}
    Assume \eqref{cond: condition general cluster size}, there exists $\mu= \mu(\kappa , \bar{b}, c_+,c_-)>0$ and $\gamma=\gamma(\kappa , \bar{b}, c_+,c_-)>0$, such that for any $A \subset V_{N,\kappa}$:
    $$
    \mathbb{E}\left[e ^{\lambda \operatorname{Card}{\widetilde{\cC}(\infty)}} \right]< \exp(\gamma|A|)\, .
    $$
\end{Lemma} 

\begin{proof} We denote by ${|\widetilde\cC(t)|}$ the cardinal of $\widetilde \cC(t)$. We observe that $\left(\widetilde{\cD}(t)\right)_{t \geq 0}$ is a classical Birth and Death process:
\begin{itemize}
    \item with rate $\upsilon_- \widetilde{\cD}(t) $, it decreases by $1$;
    \item with rate $(2m+1)\kappa\upsilon_+ \widetilde{\cD}(t)$, it increases by $4 \kappa m +2\kappa -2$ and $|\widetilde\cC(t)|$ by $4 \kappa m +2\kappa -1$;
    \item  with rate $2 \bar{b}\kappa \widetilde{\cD}(t)$, it increases by $ 4 \kappa - 2 $ and $|\widetilde\cC(t)|$ by $4 \kappa- 2$.
\end{itemize} 
Therefore, we obtain by the Kolmogorov forward equation for any $\mu,\lambda>0$
\begin{align*}
    &\partt \mathbb{E}\left[e ^{\mu |\widetilde\cC(t)|+\lambda \widetilde{\cD}(t)}\right] \\
    &=   \bbE\left[(2m+1)\kappa \upsilon_+  \widetilde{\cD}(t) \left[ e^{\mu |\widetilde\cC(t)| +\lambda \widetilde{\cD}(t)+(4\kappa m+2\kappa- 1) (\mu+\lambda) -\lambda}  - e^{\mu |\widetilde\cC(t)|+\lambda \widetilde{\cD}(t)}\right]\right]\\
    & +\bbE\left[2\bar b \kappa
    \widetilde{\cD}(t) \left[ e^{\mu |\widetilde\cC(t)|+\lambda \widetilde{\cD}(t)+(4\kappa-1)(\mu+\lambda)-\lambda}  - e^{\mu |\widetilde\cC(t)|+\lambda \widetilde{\cD}(t)}\right]\right]  \\
    &+\bbE\left[  \upsilon_-  \widetilde{\cD}(t) \left[ e^{\mu |\widetilde\cC(t)|+\lambda \widetilde{\cD}(t)- \lambda}  - e^{\mu \widetilde{\cC}(t)+\lambda \widetilde{\cD}(t)}\right]\right] \\
    &=  \bbE\left[ 
    \widetilde{\cD}(t) e^{\lambda |\widetilde\cC(t)| +2\kappa \lambda}
     \left[\upsilon_+ (e^{(4\kappa m+2\kappa-1) (\mu+\lambda) -\lambda}  - 1) + 2 \bar b \kappa (e^{(4\kappa-1) (\lambda+\mu)-\lambda}  - 1)  + \upsilon_- (e^{- \lambda }  - 1) \right]\right] 
\end{align*}
 If $\lambda$ and $\mu$ are sufficiently small, we obtain $\partt \mathbb{E}\left[e ^{\mu |\widetilde\cC(t)|+\lambda \widetilde{\cD}(t)}\right] \leq 0 $ thanks to \eqref{cond: condition general cluster size}. Consequently, $$ \mathbb{E}\left[e ^{\mu |\widetilde\cC(t)| +\lambda \widetilde{\cD}(t)}\right] \leq e^{(\lambda+\mu) |A|}\,. $$
 The result follows by taking $t$ to infinity.
\end{proof}
\subsection{Coupling the CEP with BEP} 
Our aim in this subsection is to couple the two constructions of the BEP and CEP such that we can bound the size of the information cluster. 
\begin{Lemma}
    \label{lemma: CEP coupling success}
    Given $A \subset E \subset V_{N,\kappa}$, there exists a coupling between BEP and CEP such that there exists {$\varphi: \mathscr{C}_{A}(E,\ft) \hookrightarrow \widetilde{\cC}(\infty)$} an injection. In particular, we have for any $\ft>0$:
$$ 
|\mathscr{C}_{A}(E,\ft)| \leq |\widetilde \cC(\infty)|, \qquad \text{a.s.}
$$
\end{Lemma}

\begin{proof} Fix $A \subset E \subset V_{N,\kappa}$ and $\ft >0$. Once again, we emphasise that $\fs,\ft$ refers to the time-line of the BEP, while $s,t$ refers here to the exploration time.

Let us now introduce the coupling. First, we construct the Branching Exclusion process up to time $\ft$ and denote it by $\fC_{\ft}$.

Then, we constuct $(\cC_t)_{t\ge 0}$ and a sequence of mapping $(\varphi_t)_{t \geq 0}$ as follows. At time $t=0$, we set $\varphi_0(w)=w$ for any $w \in A$.  
As long as $\tilde w \in \operatorname{Im}(\varphi_t)$ alive, it attempt to follow the evolution of $\varphi_t^{-1}(\tilde w)$ forward in time if $d(\tilde w)=1$ and backward if $d(\tilde w)=-1$:
    \begin{itemize}
        \item If $\tilde w$ (hence $\varphi_ t^{-1}(\tilde w)$) encounters a mark of $\fC_{\ft}$, we say that the mark is revealed.  Initially ($t=0$), all marks in $\fC_{\ft}$ are not revealed.
        \item If the mark encountered had already been revealed, then it is ignored. $\tilde w$ and its decendent will evolve independently of $\fC_{\ft}$.
        \item If the mark encountered has not been revealed, and it is a refresh mark. Then, $\tilde w$ is killed. In other words, we have synchronised their refresh clocks.
        \item If the mark encountered has not been revealed and it is a Glauber type $q$ at $\fs$ associated to a BEP datum $ {w}^\prime$ located $(z,j)\in\cB(\ell_{\fs}(\varphi^{-1}(\tilde w)),m) $. Then, we synchronise it with the CEP Glauber clock labelled $ (z-\ell_{\fs}(\varphi^{-1}(\tilde w)),j,q)$. Regarding the creation of data:
        \begin{itemize}
            \item \emph{If we are exploring the data in the backward direction (i.e $d(\tilde w)=-1$)}. We fix \begin{align*}
            \forall (y,k) \in B(z,m), \qquad
            \varphi_t(\mathfrak{w}(w^\prime (y-z,k)))&= \tilde w (y-\ell_{\fs}(\varphi^{-1}(\tilde w)),k)
            \end{align*}
            For a datum created in the opposite direction, for any
            $(y,k) \in B(z,m)\setminus \{(0,j)\}$, if there was a first class BEP datum $w^{\prime \prime}$ at $\fs^-$ in the position $y$. If $k\neq \kappa$, we fix
            $$\varphi_t(w^{\prime \prime})= \tilde w \# (y-\ell_{\fs}(\varphi^{-1}(\tilde w)),k)$$
            and for $k=\kappa$, we fix
               $$\varphi_t(w^{\prime \prime})= \tilde w \# (y-\ell_{\fs}(\varphi^{-1}(\tilde w)),j) \,.$$
            \item  \emph{If we are exploring the data in the forward direction (i.e $d(\tilde w)=1$)}. We fix \begin{align*}
            &\forall (y,k) \in B(z,m)\setminus \{(0,j),(0,\kappa)\}, \quad
            &\varphi_t(\mathfrak{w}(w^\prime (y-z,k)))&= \tilde w \# (y-\ell_{\fs}(\varphi^{-1}(\tilde w)),k)\\
            &\text{if } j \neq \kappa, \quad 
            &\varphi_t(\mathfrak{w}(w^\prime (y-z,\kappa)))&= \tilde w \# (y-\ell_{\fs}(\varphi^{-1}(\tilde w)),j)
            \end{align*}
            For a datum created in the same direction, for any
            $(y,k) \in B(z,m)$, if there was a first class BEP datum $w^{\prime \prime}$ at $\fs^-$ in the position $y$. We fix:
            $$\varphi_t(w^{\prime \prime})= \tilde w (y-\ell_{\fs}(\varphi^{-1}(\tilde w)),k)\,.$$
        \end{itemize}
        We say that we have discovered these data.
        \item We proceed analogously with the correlated exchange clocks.
    \end{itemize}
    
We stress that ignoring a revealed mark is without consequence for CEP's rate, due to the memoryless property of the exponential distribution. Therefore, $(\cC_t)_{t\ge 0}$ is indeed a Cluster exploration process. 

Let us notice that ignoring the revealed marks ensures that $\varphi_t$ is always injective for any $t\geq 0$. Indeed, the image of a first-class BEP datum is only fixed during the discovery of an unrevealed mark. Moreover, $\varphi_t$ is well defined for the discovered first-class BEP datum.

Moreover, Lemma \ref{lemma: information cluster control} implies that there exists almost surely a time $t^\star < \infty$ such that all CEP data are killed. We will fix:
$$\varphi= \varphi_{t^\star}\,.$$
It remains to prove the following 
\begin{Lemma}
All $\mathscr{C}_{A}(E,\ft)$ is discovered by $t_ \star$. In particular, $\varphi$ is well--defined on $\mathscr{C}_{A}(E,\ft)$. 
\end{Lemma} 
\end{proof}
Let us now prove the remaining lemma.
\begin{proof}
    Assume that there exists $w \in \mathscr{C}_{A}(E,\ft)$ that had not been discovered yet. By definition of information cluster,  there exist $w_1,\cdots w_n \in  \mathscr{C}_{A}(E,\ft) $ such that $w_1 \in A$, $w_n=w$ and for any $i \in [n]$, either $w_i \succ w_{i+1}$ or $w_{i +1}\succ w_{i}$. In other words, $\mathscr{C}_{A}(E,\ft)$ is path--connected. 
    Let $$i_\star = \inf\left\{i \leq n, \varphi \text{ is not defined for } w_i\right\}\,.$$
    In particular, no mark that induced the birth or death of $w_{i_ \star}$ in the BEP has been revealed.
   However, $w_{i_\star -1}$ was discovered at a certain time $t>0$, thanks to the reveal of a certain mark. By construction of the coupling, $\varphi_t(w_{i_\star -1})$ and $w_{i_\star -1}$ shared the same clocks. This implies that $w_i$ will be discovered, since $w_i \succ w_{i+1}$ or $w_{i +1}\succ w_{i}$. We emphasise that this clock can not have been ignored, as it is not revealed. 
\end{proof}
\subsection{Two-point correlation control}
\label{subsec Two point correlation control}
The proof of Proposition \ref{prop: covariance estimate general} follows from \cite[Claim 3.4]{univ_ising_cutoff}. For convenience of the reader, we include the proof here.
\begin{proof}
    Given $(x,i), (y,j) \in V_{N,\kappa}$. Let $\fC$, $\fC^{\prime}$ and $\fC^{\prime \prime}$ be two independent copies of the BEP dynamics started at $E= \{(x,i)\}\cup \{(y,j), j \in[\kappa]\}$.
    
    Let us denote by $A=\{(y,j), j \in[\kappa]\}$ and we introduce the event 
     $$ \cE_{(x,i)} = \left\{ (x,i) \notin \sC_A(E,\ft) \right\}\,, $$
    By definition of the information cluster, this event $ \cE_{(x,i)}$ guarentees 
    $$
    \nexists (w, w^\prime) \in \mathscr{H}_{(x,j)} (\ft)\times \sH_A(\ft)\,, \qquad\quad \fw(w)= \fw(w^\prime)\,.
    $$
    In words, there is no common first-class datum between the two update histories.
    Therefore, we may couple $\fC$ with $\fC^{\prime}$ and $\fC^{\prime \prime}$ such that on the event $\cE_{(x,i)}$:
    \begin{enumerate}
        \item $\fC(E,(x,i))=  \fC^\prime(E,(x,i)) $
        \item For any $v \in  \{(y,j), j \in[\kappa]\}$ , $\fC(E,v)=  \fC^{\prime\prime}(E,v) $
    \end{enumerate}
We obtain
$$
\begin{aligned}
\mathbb{E}\left[\bot_t(x) \bot_t(y)\right]& = \sum_{i,j=1}^\kappa \mathbb{E}\left[ \vee_t(x,i) \vee_t(y,j) \right]\\
 & = \sum_{i,j=1}^\kappa \mathbb{E}\left[ \vee_t^{\prime}(x,i) \vee_t^{\prime \prime}(y,j)+\left(\vee_t(x,i)  \vee_t(y,j) -\vee_t^{\prime}(x,i)  \vee_t^{\prime \prime}(y,j) \right) \mathds{1}_{\cE_{(x,i)}^c}\right] \\
& \leq \sum_{i,j=1}^\kappa \mathbb{E}\left[\vee_t^{\prime}(x,i) \right] \mathbb{E}\left[\vee_t^{\prime \prime}(y,j) \right]+2 \mathbb{P}\left(\cE_{(x,i)}^c\right) \\
& \leq \mathbb{E}\left[\bot_t(x) \right] \mathbb{E}\left[\bot_t(y) \right]+2 \sum_{i,j=1}^\kappa \mathbb{P}\left(\cE_{(x,i)}^c\right) .
\end{aligned}
$$
Consequently, we obtain  $$\operatorname{Cov}\left(\bot_t(x), \bot_t(y)\right) \leq 2 \sum_{i,j=1}^\kappa \mathbb{P}\left(\cE_{(x,i)}^c\right)\,. $$
Finally, we observe that
$$
\sum_{x=1}^N \operatorname{Cov}\left(\bot_t(x), \bot_t(y)\right) \leq 2  \sum_{x=1}^N \sum_{i,j=1}^\kappa  \Prob{(x,i) \in \sC_A(E,\ft)} \leq 2 \kappa \bbE\left[\left|\sC_A(E,\ft)\right|\right]\,,
$$
Lemma \ref{lemma: CEP coupling success} and 
Lemma \ref{lemma: information cluster control} implies the result. The same proof holds for $\top$.
\end{proof}

\section{Proof of Lemma  \ref{lemma: RHS- one red control}}
\label{section: Proof of Propostion the control of one error and lcocalisation}
This section   proves Lemma \ref{lemma: RHS- one red control}. While the idea of the proof is close to the techniques developped in \cite{HQ}, they  need to be adapted to our framework due to the change of the underlying graph and the added correlating jumps. To this end, we recall the graphical construction of $(\mathpzc B, U, V)$ explained in Subsection \ref{subsec: A single red perturbation} and we recall that $U$ and $V$ are two independent tracers, conditionnally on $\widehat \Xi$. 

For convenience, we say that $\mathpzc B ^\prime \in \{0,1\}^{\bbG_{N,\kappa}}$ is $D$-dense if and only if $$ \max_{ u \in V_{N,\kappa}} \operatorname{dist}(u, \mathpzc B^\prime) \leq D\,,$$
where we  identified the configuration $\mathpzc{B} \in \{0,1\}^{\bbG_{N,\kappa}}$ and the set of blue positions $ \{v \in V_{N,\kappa}, \, \mathpzc{B}_v=1 \, \}$. We denote by $\cD(D, t)$ the event on which $ (\mathpzc B _s)_{s \leq t}$ remains $D$-dense on $  [0, t]$
$$
\cD(D, t)= \{ \forall s \in [0, t],  
        \max_{ u \in V_{N,\kappa}} \operatorname{dist}(u, \mathpzc B _s) \leq D\}\,.
$$

The proof is now divided into three main steps. First, we observe that Lemma \ref{lemma: RHS- one red control} boils down to controlling the amount of time the two tracers $U$ and $V$ meet, before getting killed.
\begin{Lemma}
    \label{lemma: RHS- one red control- step 1}
    For any $t>0$, there exists $c_6(t)>0$ such that 
    \begin{equation}
        \label{eq : RHS- one red control- step 1}
      \begin{split} 
        \mathbb{E}&\left[ \sum_{\substack{u \in V_{N,\kappa}\,,\\s \in \Xi^G_u \cap [0,t] }} \ind_{U_{s}=V_{s}\in B(u,m)} + \sum_{\substack{e \in E_{N,\kappa}\,,\\s \in \Xi^C_e \cap [0,t] }} \ind_{U_{s}=V_{s}\in e}+\ind_{U_{t}=V_{t}} \right]\\ &\qquad\qquad\leq c_6(t) \int_0^t \Prob{U(s)=V(s)\neq \dagger} \dd s \,.
      \end{split}
    \end{equation}
\end{Lemma}

Thanks to the particular structure of $\bbG_{N,\kappa}$, we focus on bounding
$$
\int_0^t \Prob{\operatorname{dist}(U(s),V(s))=0}\dd s= \Esp{\int_0^t \ind_{\operatorname{dist}(U(s),V(s))=0 } \dd s}\,.
$$
where $\operatorname{dist}$ is defined in \eqref{def: distance on the expanded graph} and we choose as a convention $\operatorname{dist}(\dagger,u)=\operatorname{dist}(u,\dagger)=\infty$ for any $u\in \bbG_{N,\kappa}\cup \{\dagger\}$. Let us introduce some notations 
\begin{itemize}
    \item For any $i\geq 0$, $\cT_i$ the time of the $i^{\rm th}$ return of  $\operatorname{dist}(U,V)$ to zero.
    \item For any $i\geq 0$ such that $\cT_i< \infty $, we denote by $\tau_i^{0 \rightarrow}$ the time $\operatorname{dist}(U,V)$ remained equal to zero after $\cT_i$
   $$
   \tau_i^{0 \rightarrow}= \inf\{s \geq 0, \operatorname{dist}(U(s+\cT_i ),V(s+\cT_i)) \neq 0 \}\,. 
   $$
    \item For any $i\geq 0$ such that $\cT_i< \infty $, we denote by $\tau_i^{0 \rightarrow 1}$ the  first time $\operatorname{dist}(U,V)$ reached the value $1$ after $\cT_i$
   $$
   \tau_i^{0 \rightarrow 1}= \inf\{s \geq 0, \operatorname{dist}(U(s+\cT_i ),V(s+\cT_i)) = 1 \}\,. 
   $$
   Notice that $\tau_i^{0 \rightarrow } \leq \tau_i^{0 \rightarrow 1}$, since one of tracer may be killed.
   \item For any $i\geq 0$ such that $\cT_i$ and $\tau_i^{0 \rightarrow 1}< \infty $, we denote by $\tau_i^{1 \rightarrow 0}$ the first time $\operatorname{dist}(U,V)$ needed to reach back the value $0$ after $\cT_i+ \tau_i^{0 \rightarrow 1}$
   $$
   \tau_i^{0 \rightarrow 1}= \inf\{s \geq 0, \operatorname{dist}(U(s+\cT_i+ \tau_i^{0 \rightarrow 1}),V(s+\cT_i+\tau_i^{0 \rightarrow 1})) = 0 \}\,. 
   $$
\end{itemize}
By convention, once $U$ or $V$ is killed, all these hitting times are set equal to $\infty$ for convenience.

Thanks to these notations, we decompose the interval $[0,t]$ into excursions of the process $\operatorname{dist}(U,V)$ and we obtain
\begin{equation}
    \label{eq: intermediate step 1 RHS proof excursion}
    \int_0^t \ind_{\operatorname{dist}(U(s),V(s))=0 } = \sum_{i\geq 0} \ind_{\cT_i< t } \min(\tau_i^{0\rightarrow}, t-\cT_i) \,.
\end{equation}
 Therefore, it is sufficient to control both the time needed to exit zero and the number of excursions. 
\begin{Lemma}[Bound on the exiting time]
    \label{lemma : RHS step 2 Bound on the exiting time}
    Given a sequence $(D_N)_{N \in \N}$ such that $ {D_N^2}{=o(\beta_N)}$. There exists $c_7=c_7(t)>0$ such that for any $N \in\N$, $\mathpzc B_0 \in \{0,1\}^{\bbG_{N,\kappa}}$, $D_N$-dense, and $U_0,V_0 \in V_{N,\kappa}$ such that $\operatorname{dist}(U_0,V_0)=0$
    \begin{equation}
        \max_{t\geq 0}\bbE_{\mathpzc B_0, U_0,V_0}\left[\min(\tau_1^{0 \rightarrow},t) \mathds{1}_{\cD(D_N,t)} \right] \le c_7 \frac{D_N^2}{\beta_N} \, .
    \end{equation}
\end{Lemma}
 This result is expected. Indeed, if we assume that the tracers are not killed, the only way to separate them is for them to reach a blue position and then make a different choice. Therefore, we expect that we need to wait at least a time $\frac{D_N^2}{\beta_N}$.

\begin{Lemma}[Bound on the number of excursion]
    \label{lemma: RHS step 3 Bound on the number of excursion}
    For any $t\ge 0$, there exist $c_{9}=c_{9}(t)>0$ such that 
    $$
    \bbE\left( \sum_{i \geq 0} \mathds{1}_{{\cT}_i \le t}\right) \leq c_{9}  \sqrt{\beta_N} .
    $$
\end{Lemma}
Now, we can prove Lemma \ref{lemma: RHS- one red control}. 
\begin{proof}[Proof of Lemma \ref{lemma: RHS- one red control}]
    Thanks to \eqref{eq: intermediate step 1 RHS proof excursion}, we have
    \begin{align*}
        \Esp{\int_0^t \ind_{\operatorname{dist}(U(s),V(s))=0 } \dd s} & \leq t \Prob{\cD(D_N,t)}+  \sum_{i=1}^\infty \Esp{  \ind_{\cT_i< t } \ind_{ \cD(D_N,t) } \max(\tau_i^{0\rightarrow}, t-\cT_i)}\\
        & \leq t \Prob{\cD(D_N,t)}+ c_7 \frac{D_N^2}{\beta_N}  \sum_{i=1}^\infty \Esp{  \ind_{\cT_i< t } } \\
        & \leq t \Prob{\cD(D_N,t)}+ c_7 c_9 \frac{D_N^2}{\sqrt{\beta_N}}\,,
    \end{align*}
    where in the second inequality, we used the Markov property and Lemma \ref{lemma : RHS step 2 Bound on the exiting time}, and in the last inequality we used Lemma \ref{lemma: RHS step 3 Bound on the number of excursion}. Finally, Lemma \ref{lemma: RHS- one red control- step 1} yields the result.
\end{proof}
\subsection{Proof of Lemma \ref{lemma : RHS step 2 Bound on the exiting time}} Let us fix a sequence $(D_N)_{N \in \N}$ such that $ {D_N^2}{=o(\beta_N)}$. 
We claim that it is sufficient to prove the following lemma.
\begin{Lemma}
    \label{lemma: bound exiting time _ step 1}
There exist $c_8>0$ such that for any $\mathpzc B_0 \in \{0,1\}^{\bbG_{N,\kappa}}$, $D_N$-dense, and $U_0,V_0 \in V_{N,\kappa}$ such that $\operatorname{dist}(U_0,V_0)=0$, we have 
\begin{equation}
    \label{eq: bound exiting time _ step 1}
    \bbE_{\mathpzc B_0, U_0,V_0}\left(\tau_0^{0 \rightarrow} \leq  8\frac{D_N^2}{\beta_N} \right) \geq c_{8}>0\,.
\end{equation}
\end{Lemma} 

By splitting the time interval $[0,t]$ in interval of size $\frac{8 D_ N ^2}{\beta_N}$, Lemma \ref{lemma : RHS step 2 Bound on the exiting time} boils down to studying a geometric law of parameter $c_8$ (see \cite[Proof of Lemma 43]{HQ}). Let us now focus on the proof of   Lemma \ref{lemma: bound exiting time _ step 1}.
\begin{proof}[Proof of Lemma \ref{lemma: bound exiting time _ step 1}] We observe that if $U_0 \neq V_0$, then $\tau_0^{0 \rightarrow}$ is the first time either $U$ or $V$ jumps thanks to an exclusion mark or is killed. Therefore, in this case, we obtain
\begin{equation}
    \tau_0^{0 \rightarrow} \preceq {\rm Exp}( 4 \kappa \beta_N )
\end{equation}
where $\preceq $ corresponds to stochastic domination. This imply \eqref{eq: bound exiting time _ step 1}. However, if $U_0 = V_0$, both $U$ and $V$ are stuck together until they are killed or meet another blue position. We will focus on the second opportunity. 

Since  $\mathpzc B_0$ is $D_N$-dense, there exist $W_0 \in V_{N,\kappa}$ such that 
$$1 \leq  \operatorname{dist}(U_0,W_0) \leq 2 D_N \,.$$
Let us then consider the enhanced expanded process\footnote{see Remark \ref{remark: multiple tracers}.} with three tracers $U,V$ and $W$ started respectively at $U_0,V_0$ and $W_0$ and evolving independently conditionally $\widehat{\Xi}$. We denote by $\tau$ the first time $\operatorname{dist}(U,W) = 1$ and $(\cG_t)_{t\ge 0}$ the associated filtration. 

As long as neither $U$ nor $W$ is killed and  $\operatorname{dist}(U,W) \geq 1$, $\iota(U)$ and $\iota(W)$ follow independent random walks on the original graph $\T_{N}$. Therefore, there exists $c_{11}>0$ such that
\begin{equation}
     \label{eq: bound exiting time _ step 2}
    \bbP\left(\left.\tau \leq \frac{8 D_N^2}{\beta_N}\right|
     W_{\frac{8 D_N^2}{\beta_N}} \neq \dagger \right) \ge c_{11}\,,
\end{equation} thanks to  \cite[Lemma 47]{HQ}. Moreover, we observe that there exists a constant $C=C(c_-,c_+,\tilde{b},\kappa)$ such that 
\begin{equation*}
   \bbP\left[ 
   W_{\frac{ 8 D_N^2}{\beta_N}} \neq \dagger \right] \geq \Prob{\cP\left( \frac{C D_N^2}{\beta_N}\right) =0 }\,.
\end{equation*}
Since 
$
\frac{8 D_N^2}{\beta_N} 
$ converge to zero,  we have for $N$ large enough, 
\begin{equation}
     \label{eq: bound exiting time _ step 3}
   \bbP\left[ 
    W_{\frac{8D_N^2}{\beta_N}} \neq \dagger \right] \geq \frac 1 2\,.
\end{equation}
Moreover, at $\tau$, the two random walks $U$ and $V$ are neighbours of $W$. Therefore, $U$ and $V$ can separate if there is a hesitating exclusion mark on the edge between $U_{\tau}$ and $W_{\tau}$ before any one of them moves away due to an exclusion or a hesitating exclusion mark. Moreover,   $U$ and $V$ need to make different choices. Thus, conditionnally on $\cG_\tau$ and $ W_{\frac{8D_N^2}{\beta_N}} \neq \dagger$, we have 
\begin{equation}
     \label{eq: bound exiting time _ step 4}
    \bbP\left[\left.\tau_0^{0 \rightarrow} \leq \tau +\frac{1}{\beta_N}\right|\cG_\tau,  W_{\frac{8D_N^2}{\beta_N}} \neq \dagger \right] 
   \geq \Prob{P_1\geq 1,\, P_2 = 0, \, P_3 =0 , S_U \neq S_V}\, , 
\end{equation}
where $P_1, P_2,P_3, S_U$ and $S_V$ are independents, $P_2,P_3$ (resp. $P_1$) are Poissons variables of parameter $3\beta_N\kappa \frac{1}{\beta_N}$ (resp. $2\beta_N \frac{1}{\beta_N}$  ) and $S_U$ and $S_V$ are Binomial random variables of parameter  $\frac{1}{2}$. Notice that the right hand side of \eqref{eq: bound exiting time _ step 4} is a positive constant.
The proof then follows from \eqref{eq: bound exiting time _ step 2},\eqref{eq: bound exiting time _ step 3} and \eqref{eq: bound exiting time _ step 4}.
\end{proof}
\subsection{Proof of Lemma \ref{lemma: RHS step 3 Bound on the number of excursion}}
We are interested in studying the behaviour of $\operatorname{dist}(U, V)$. Let us sum up the rate at which this distance evolves as long as neither $U$ nor $V$ is killed:
\begin{itemize}
    \item If $\operatorname{dist}(U,V)\ge 2$, then $U$ and $V$ behave as two independent random walk on $\bbG_{N,\kappa}$ where all edges have rates $ \beta_N$. Therefore, $\iota(U)$ and $\iota(V)$ behave as two independent random walks on $\T_{N}$ where all edges have rates $\kappa \beta_N$. 
    \item If $\operatorname{dist}(U,V)= 1$: it will decrease by $1$ with rate $(2\kappa-2)\beta_N+\beta_N$ or increase by $1$ with rate $2\kappa \beta_N$ 
    \item If $\operatorname{dist}(U,V)= 0$: the rate depend on the position of $U$ and $V$ as well as the blue positions $\mathpzc B$
    \begin{itemize}
        \item If $U=V$, it increases with rate $ \operatorname{Card}(\mathpzc B \cap \left[B(U,1)\setminus B(U,0)\right]) \beta_N $.
        \item If $U \neq V$, it increases with rate $ 4 \kappa \beta_N $.
    \end{itemize}
\end{itemize}

We also introduce $(\hat U, \hat V)$ two independent random walk on $\bbG_{N,\kappa}$ accelerated by $\beta_N$. We observe that the rates of $\operatorname{dist}(\hat U, \hat V)$ changes always dominate the rate $\operatorname{dist}( U,  V)$. Therefore, we can compare their excursions' length.
\begin{Lemma}
    \label{lemma: bounds on the number of excursion}
    Given 
    $$
    \operatorname{dist}( U,  V)(0)=\operatorname{dist}(\hat U, \hat V)(0)=0\,,
    $$
    There exists a coupling such that almost surely for any $i\geq 0$
    \begin{equation}
        \tau_i^{0 \rightarrow 1} \geq \hat{\tau}_i^{0 \rightarrow 1} \quad \text{and} \quad \tau_i^{1 \rightarrow 0} \geq \hat{\tau}_i^{1 \rightarrow 0}\,. 
    \end{equation}
    In particular, we have almost surely for any $t\geq 0$, 
    $$
\sum_{i=0}^\infty \ind_{\cT_i \leq t} \leq \sum_{i=0}^\infty \ind_{\hat{\cT}_i \leq t}\,.
    $$
\end{Lemma}
Lemma \ref{lemma: RHS step 3 Bound on the number of excursion} is then a direct consequence of Lemma \ref{lemma: bounds on the number of excursion} and the following lemma, which controls the expected number of excursions of $\operatorname{dist}(\hat U, \hat V)$, thanks to the local central limit theorem.
\begin{Lemma}[{\cite[Lemma 49]{HQ}}]
   For any $t\ge 0$, there exists $c_{10}=c_{10}(t)>0$ such that  
    $$
    \bbE\left( \sum_{i \geq 0} \mathds{1}_{\hat{\cT}_i \le t}\right) \leq c_{10} \sqrt{\beta_N}.
    $$
\end{Lemma}
We finish this section by proving Lemma \ref{lemma: bounds on the number of excursion}. 
\begin{proof}[Proof of Lemma \ref{lemma: bounds on the number of excursion}]
Since $(\mathpzc B, U, V)$ is a Markov process, it is sufficient to prove that there exists a coupling such that almost surely 
\begin{equation}
        \label{eq : exursion inequality step 1}
        \tau_0^{0 \rightarrow 1} \succeq \hat{\tau}_0^{0 \rightarrow 1} \quad \text{and} \quad \tau_0^{1 \rightarrow 0} \succeq \hat{\tau}_0^{1 \rightarrow 0}\,,
\end{equation}
as long as $  \operatorname{dist}( U, V)(0)=\operatorname{dist}(\hat U, \hat V)(0)=0$. Let us first focus on the first inequality. First, observe that if $U$ and $V$ are killed before their first separation, the RHS inequality in \eqref{eq : exursion inequality step 1} is trivial. We distinguish between the two following cases 
\begin{itemize}
    \item If $U(0)\neq V(0)$, either $\tau_0^{0 \rightarrow 1}=\infty$ because $U$ and $V$ are killed or it is equal to the first time $U$ or $V$ encounter an exclusion or Hesitating exclusion mark. Therefore, $$\tau_0^{0 \rightarrow 1}  \overset{\cL}{=} \infty \ind_{U(\tau^{0\rightarrow})= \dagger} + S \ind_{U(\tau^{0\rightarrow})\neq \dagger}$$
    where $S \sim {\rm Exp}(4 \kappa \beta_N)$. The result follows since  $\hat\tau_0^{0 \rightarrow 1} \sim {\rm Exp}(4 \kappa \beta_N)$.
    \item  If $U(0) =  V(0)$, the only way for $U$ and $V$ to separate is to encounter an Hesitating exclusion mark such that the opposite vertex is a blue position. Moreover, one needs to succeed and the other to fail. Therefore, by dropping the condition on the opposite vertex, we obtain 
    $$\tau_0^{0 \rightarrow 1} \succeq \tilde \tau$$
    where $\tilde \tau$ is the first time $U$ and $V$ encounter a Hesitating exclusion mark and make different choices or is killed. We notice that 
     $$\tilde \tau \overset{\cL}{=} \infty \ind_{U(\tau^{0\rightarrow})= \dagger} + S \ind_{U(\tau^{0\rightarrow})\neq \dagger}$$
   where $S\overset{\cL}{=} {\rm Exp}(2\kappa \beta_N)$.
     Whereas, we have $ \hat\tau_0^{0 \rightarrow 1} \sim  {\rm Exp}(4 \kappa\beta_N)$. 
\end{itemize}  
Let us now prove the second inequality in \eqref{eq : exursion inequality step 1}.   We notice that for $\operatorname{dist}( U, V)$ (resp. $\operatorname{dist}(\hat U, \hat V)$),to reach state zero, 
\begin{itemize}
    \item It first waits an exponential law with parameter $(4\kappa - 1)\beta_N$ (resp.  $4\kappa\beta_N$) to see a change.
    \item Either, we successfully decrease with probability $\frac{2\kappa- 1}{4\kappa-1}$ (resp. $\frac 1 2$)
    \item Or, it increases. In this case, we need to wait for the distance to become again equal to $1$, which will take time $\tau_0^{2 \rightarrow 1}$ (resp. $\hat \tau_0^{2 \rightarrow 1}$) and then, attempt again.
\end{itemize}
Finally, we observe that 
$$
  \tau_0^{2 \rightarrow 1} \overset{d}{=} \hat{\tau}_0^{2 \rightarrow 1}
$$
and we have the following stochastic domination $$ \cB\left(\frac{1}{2} \right) \succeq \cB\left(\frac{2\kappa- 1}{4\kappa-1} \right)  \quad \text{and} \quad {\rm Exp}\left( (4\kappa-1)\beta_N \right) \succeq {\rm Exp}\left(4\kappa\beta_N \right)\,.$$ This concludes the proof. 
\end{proof}

\subsection{Proof of  Lemma \ref{lemma: RHS- one red control- step 1}}
We bound each term of the RHS in \eqref{eq : RHS- one red control- step 1} separately.
Let us observe that both
\begin{align*}
    \sum_{\substack{u \in V_{N,\kappa}\,,\\s \in \Xi^G_u \cap [0,t] }} \ind_{U_{s}=V_{s}\in B(u,m)} \qquad \text{and} \qquad \sum_{\substack{e \in E_{N,\kappa}\,,\\s \in \Xi^C_e \cap [0,t] }} \ind_{U_{s}=V_{s}\in B(e,0)}
\end{align*}
are sums over marks of Poisson Point Processes. Therefore, Campbell formula \cite[Proposition 2.7]{LP18} implies 
\begin{align*}
    \bbE\left(\sum_{\substack{u \in V_{N,\kappa}\,,\\s \in \Xi^G_u \cap [0,t] }} \ind_{U_{s}=V_{s}\in B(u,m)} \right) &\leq \upsilon_+  \int_0^t \sum_{u \in V_{N,\kappa}} \Esp{\ind_{U_{s}=V_{s}\in B(u,m)}}  \dd s\\
   & \leq (2m+1)\kappa \upsilon_+  \int_0^t  \Esp{\ind_{U_{s}=V_{s}\neq \dagger}} \dd s\,,
\end{align*}
where $\upsilon_+ =\sum_{j=-q_-}^{q_+} \lambda_j$. Similarly, we obtain 
\begin{align*}
    \bbE\left(\sum_{\substack{e \in E_{N,\kappa}\,,\\s \in \Xi^C_e \cap [0,t] }} \ind_{U_{s}=V_{s}\in e}\right) &\leq \bar b  \int_0^t \sum_{e \in E_{N,\kappa}} \Esp{\ind_{U_{s}=V_{s}\in B(e,m)}}  \dd s \leq 2 \kappa \bar b \int_0^t  \Esp{\ind_{U_{s}=V_{s}\neq \dagger}} \dd s 
\end{align*}
Finally, we observe that $\Prob{U_t=V_t\neq \dagger \rvert \widehat \Xi}$ is non-increasing, where $\widehat \Xi$ is the common background process defined in Subsection \ref{subsec: A single red perturbation}.

This implies
\begin{equation}
    \Prob{U_t=V_t\neq \dagger} \leq \frac{1}{t} \int_0^t  \Prob{U_s=V_s\neq \dagger}\,,
\end{equation}
which yields Lemma \ref{lemma: RHS- one red control- step 1}.

\begin{appendices}
    \section{Control of $\sR$}
    \label{section-appendix: control of Red}
    In this section, we are interested in the expected number of red positions. We recall that we defined the red position in Subsection \ref{subsection: RGB}  as 
    \begin{align*}
         \mathscr{R}(\ft)
         &=\left\{v \in V_{N,\kappa}; \text{ such that } \quad  \fD_{\ft} \cap \mathscr{H}_v(\ft) \neq \emptyset \right\}\\
         &=\{u \in V_{N,\kappa}; \text{ such that } \quad \wedge_ \ft(u) \neq \vee_\ft(u)\},.
    \end{align*}
    We denote by 
    \begin{equation}
        \mathpzc{r}_\ft=\operatorname{Card}(\mathscr{R}(\ft)) = \sum_{u \in V_{N,\kappa}} \wedge_\ft(u) -  \vee_\ft(u) \, .
    \end{equation}
    Let us start by proving  Lemma \ref{lemma: control of REDs exp control}. 
    \begin{proof}[Proof of Lemma \ref{lemma: control of REDs exp control}]
        Fix $c_1,c_2>0$ that we will explicitly later.
        It is sufficient to prove for any $u \in V_{N,\kappa}$
       $$  e^{-c_1 \ft} \, \Esp{\wedge_\fs(u) -  \vee_\fs(u)} \, \leq \,\Esp{\wedge_{\ft+ \fs}(u) -  \vee_{\ft+ \fs}(u)}\, \leq \, e^{-c_2 \ft }  \, \Esp{\wedge_\fs(u) -  \vee_\fs(u)} \,.$$
     where $\vee, \wedge$ are constructed as in \eqref{eq: vee and wedge construction}. 
     Regarding the lower bound, we construct two independent BEP $(\fC^\prime_ {\fu})_{\fu \in  [0,\fs]}$ and $(\fC_ {\fu})_{\fu \in  [0,\ft+\fs]}$. 
     
     On the event $u$ is not killed during the $(0,\ft)$ in $(\fC_ {\fu})_{\fu \in  [0,\ft]}$, we have 
        $$
        \wedge_{\ft+\fs}(u) \overset{\cL}{=} \wedge^\prime_{\fs}(\ell_\ft(u)) \quad \text{and} \quad  \vee_{\ft+\fs}(u) \overset{\cL}{=} \vee^\prime_{\fs}(\ell_\ft(u))\,,
        $$
    where $(\wedge^\prime,\vee^\prime)$ are constructed through $(\fC^\prime_ {\fu})_{\fu \in  [0,\fs]}$, $\ell_\ft(u)$ is the position of the datum $u$ in $(\fC_ {\fu})_{\fu \in  [0,\ft]}$.   Notice also that $u$ is killed with rate at most $ 2 \bar b \kappa+ (2m+1)\kappa\upsilon_+ +\upsilon_-$, which imply 
    $$
    \Prob{{\rm Exp}( 2 \bar b \kappa+ (2m+1)\kappa\upsilon_+ +\upsilon_-) > \ft} \min_{v}\Esp{\wedge^\prime_\fs(v) -  \vee^\prime_\fs(v)} \leq \Esp{\wedge_{\ft+ \fs}(u) -  \vee_{\ft+ \fs}(u)}\,.
    $$
    This concludes the proof of the lower bound with $c_1= 2 \bar b + (2m+1)\kappa\upsilon_+ + \upsilon_-$. We now turn to the proof of the upper bound.  The expression of the expanded process generator \eqref{def-eq: generator of the exapnded process} implies 
    \begin{align*}
        \partial_\ft \Esp{\wedge_\ft(0)-\vee_ \ft(0)}= & - \sum_{j= -q_- -1} ^{-1}  \lambda_ j \bbE\left[\wedge_ \ft(0)f_j(\wedge_\ft) - \vee_\ft(0)f_j(\vee_\ft)\right]\\
        & +  \sum_{j= 1 } ^{q_++1}  \lambda_j \bbE\left[(1 -\wedge_\ft(0))f_j(\wedge_\ft) - (1-\vee_ \ft(0))f_j(\vee_ \ft)\right]\,
    \end{align*}
    where we used the transitivity of the graph $\bbG_{N,\kappa}$ to get rid of the terms associated with the misanthrope parts $\cA_{C,L}$ and $\cA_{C,R}$. Let us consider each term seperately
    \begin{itemize}
        \item For $j=-q_--1$, we have 
         $$\bbE\left[\wedge_ \ft(0)f_j(\wedge_\ft) - \vee_\ft(0)f_j(\vee_\ft)\right] =  \Esp{\wedge_\ft(0)-\vee_ \ft(0)}$$
        \item Similarly, for $j=-q_++ 1$, we have 
         $$\bbE\left[(1 -\wedge_\ft(0))f_j(\wedge_\ft) - (1 -\vee_\ft(0))f_j(\vee_\ft)\right] =  - \Esp{\wedge_\ft(0)-\vee_ \ft(0)}\,.$$
        \item  For $0>j>-q_--1$, we have
        $$\bbE\left[\wedge_ \ft(0)f_j(\wedge_\ft) - \vee_\ft(0)f_j(\vee_\ft)\right] \geq   \Esp{\ind_{\wedge_ \ft(0)=\vee_\ft(0)=1}(f_j(\wedge_\ft) - f_j(\vee_\ft))}\,,$$
        because $\wedge_\ft(0) \geq \vee_\ft(0)$, almost surely. 
        Notice also 
        $$
        \Esp{\ind_{\wedge_ \ft(0)=\vee_\ft(0)=1}(f_j(\wedge_\ft) - f_j(\vee_\ft)) }\geq - ((2m+1)\kappa- 1) \Esp{(\wedge_\ft(0) - \vee_\ft(0))}\,.
        $$
        \item For $0<j<-q_++1$, we have
        $$\bbE\left[(1 -\wedge_\ft(0))f_j(\wedge_\ft) - (1 -\vee_\ft(0))f_j(\vee_\ft)\right] \leq  \Esp{\ind_{\wedge_ \ft(0)=\vee_\ft(0)=0}(f_j(\wedge_\ft) - f_j(\vee_\ft))}\,,$$
       because $\wedge_\ft(0) \geq \vee_\ft(0)$, almost surely. Notice also 
        $$
        \Esp{\ind_{\wedge_ \ft(0)=\vee_\ft(0)=0}(f_j(\wedge_\ft) - f_j(\vee_\ft)) }\leq ((2m+1)\kappa- 1) \Esp{(\wedge_\ft(0) - \vee_\ft(0))}\,.
        $$
    \end{itemize}
    In the end, we obtain
    \begin{align*}
        \partial_\ft \Esp{\wedge_\ft(0)-\vee_ \ft(0)}\leq & [- \upsilon_- + ((2m+1)\kappa-1) \upsilon_+ ] \Esp{\wedge_\ft(0)-\vee_ \ft(0)}\,.
    \end{align*}
   { This implies the upper-bound for $c_2=\upsilon_- - ((2m+1)\kappa-1) \upsilon_+$ thanks to \eqref{cond: trivial cluster size} or \eqref{cond: condition general cluster size}.}
    \end{proof}

    Finally, we prove Lemma \ref{lemma: technical lemma on R}. 

    \begin{proof}[Proof of Lemma \ref{lemma: technical lemma on R}]
        Let $o \in\{0,1\}$. By doing the same computation as in the previous proof and using attractivity, we obtain
        $$
    \forall \alpha \in [0,\kappa] , \,\,\, R^\prime(\alpha) < 0\,.
    $$
    Therefore, $R$ admits a unique zero on $[0,\kappa]$ denoted by $\alpha_\star$. This implies that $\alpha_o$ converges to $\alpha_\star$. Now, we also have:
    \begin{align*}
        \frac{\partt (\alpha_o -\alpha_\star)}{\alpha_o -\alpha_\star}= \frac{R(\alpha_o)- R(\alpha_\star)}{\alpha_o -\alpha_\star} \,,
    \end{align*}
    which converge to $R^\prime(\alpha_\star)$ as $t$ goes to infinity. The result then follows from the Gronwall Lemma.
    \end{proof}
    
    \section{Anticoncentration lemma}
    \label{section anticoncentation}
    In this section, we are interested in $(U_1,U_2)$, the occupied position for an exclusion process started with two particles at $u_1$ and $u_2$. 
    
    \begin{Lemma}
        \label{antioncentration lemma} 
        Let $\vartheta>0$ and $\zeta \sim {\rm Exp}\left(\vartheta\right)$,
        independent of $(U_1,U_2)$. Then, for any $k \in \mathbb{Z}_{+}$,
    $$
\max _{u_1, u_2} \mathbb{P}_{u_1, u_2}\left[\operatorname{dist}\left(U_1(\zeta), U_2(\zeta)\right) \leq k\right] \leq c_{\kappa} \left[\frac{ N}{\beta_N}+ \vartheta \frac{1}{N}\right]\,.
$$
\end{Lemma}
\begin{proof}
As in \cite[Lemma 21]{HQ}, we couple the evolution $(U_1,U_2)$ with two independent random walk  $(\Hat{U_1},\Hat{U_2})$ such that almost surely for any $s\geq0$,
$$ \operatorname{dist}\left(U_1(s), U_2(s)\right)  \geq \operatorname{dist}\left(\hat U_1(s),\hat U_2(s)\right)\,. $$

To this end, it is sufficient to specify the expression of the generator on the configuration  satisfying $\operatorname{dist}\left(u_1, u_2\right)  \geq \operatorname{dist}\left(\hat u_1,\hat u_2\right)$
\begin{itemize}
    \item For $\operatorname{dist}\left(u_1, u_2\right)  > \operatorname{dist}\left(\hat u_1,\hat u_2\right)$, the processes evolve independently.
    \item For $\operatorname{dist}\left(u_1, u_2\right)  = \operatorname{dist}\left(\hat u_1,\hat u_2\right)$. Let $u_1=(x_1,j_1)$ and $u_2=(x_2,j_2)$ (and similarly, for $\hat u_1$ and $\hat u_2 $ ). We observe that there exist $z\in\{\pm 1\}$ such that $x_1-x_2= z (\hat x_1- \hat x_2)$. We fix the jump rates as follows
    \begin{itemize}
        \item If $\operatorname{dist}\left(u_1, u_2\right)  = \operatorname{dist}\left(\hat u_1,\hat u_2\right)\neq 1$\footnote{In particular, it includes the case $\operatorname{dist}\left(u_1, u_2\right)  = \operatorname{dist}\left(\hat u_1,\hat u_2\right)= 0$}, we choose for any $j \in[\kappa]$  
    \begin{align*}
          &\cA^{\rm couple}(((u_1,u_2),(\hat u_1,  \hat u_2)); ((u_1\pm (1,j),u_2), (\hat u_1 \pm z (1,j), \hat u_2)))\\
          &=   \cA^{\rm couple}(((u_1,u_2),(\hat u_1,  \hat u_2)); ((u_1,u_2\pm (1,j)), (\hat u_1, \hat u_2\pm z (1,j)))) = \beta_N
    \end{align*}
   where for $u=(x,i)$, we used the notation $u\pm (1,j)= (x\pm 1,j)$.
        \item If $\operatorname{dist}\left(u_1, u_2\right)  = \operatorname{dist}\left(\hat u_1,\hat u_2\right)=1$, let us assume $x_2=x_1+1$. We choose for any $j \in[\kappa] \setminus \{j_2\}, l \in [\kappa]$ and $k \in   [\kappa] \setminus \{j_1\}$
    \begin{align*}
          &\cA^{\rm couple}(((u_1,u_2),(\hat u_1,  \hat u_2)); ((u_1 + (1,j),u_2), (\hat  u_1 + z (1,j), \hat u_2)))\\
          &=   \cA^{\rm couple}(((u_1,u_2),(\hat u_1,  \hat u_2)); ((u_1,u_2-  (1,k)), (\hat u_1, \hat u_2- z (1,k)))) \\
           &=\cA^{\rm couple}(((u_1,u_2),(\hat u_1,  \hat u_2)); ((u_1 - (1,l),u_2), (\hat u_1 - z (1,l), \hat u_2)))\\
          &=   \cA^{\rm couple}(((u_1,u_2),(\hat u_1,  \hat u_2)); ((u_1,u_2 + (1,l)), (\hat u_1, \hat u_2+ z (1,l))))= \beta_N
    \end{align*}
    We also impose
    \begin{align*}
          & \cA^{\rm couple}(((u_1,u_2),(\hat u_1,  \hat u_2)); (u_1,u_2), (\hat  u_1 , \hat u_1))= \beta_N\\
           & \cA^{\rm couple}(((u_1,u_2),(\hat u_1,  \hat u_2)); (u_1,u_2), (\hat  u_2 , \hat u_2))= \beta_N\,.
    \end{align*}
    We proceed similarly if $x_1=x_2+1$.
    \end{itemize} 
\end{itemize}
One can check that such a coupling exists and preserve the order of the distances. We omit its expression because it is fastidious.
Finally, we observe that $(\operatorname{dist}(\hat U_1(s), \hat U_2(s)))$ has the same law as $(\operatorname{dist}_ {\T_N}(X_1(\kappa s), X_2(\kappa s)))$ where $X_1$ and $X_2$ are two independent random walk on $\T_N$ where every edge has conductance $\beta_N$. Therefore, the anticoncentration of the random walk \cite[Lemma 51]{HQ} implies the result.
\end{proof}
\end{appendices}

\section{Acknowledgments} 
The author thanks Max Fathi for his guidance and advices throughout this project and for carefully reading several versions of the manuscript. The author also warmly kindly thanks Cyril Labbé, Ellen Saada, Justin Salez and Hong Quan Tran for helpful discussions. 

This work received funding from the Agence Nationale de la Recherche (ANR) Grant
ANR-23-CE40-0003 (Project CONVIVIALITY), as well as funding from the Institut Universitaire de France.

An AI assistant was used \emph{solely} for language editing (grammar, spelling and phrasing). The author takes full responsibility for the content of this article.
\bibliographystyle{plain}
\bibliography{biblio}

\begin{thebibliography}{10}

\bibitem{Aldous1986SHUFFLINGCA}
D.~Aldous and P.~Diaconis.
\newblock Shuffling cards and stopping times.
\newblock {\em Amer. Math. Monthly}, 93(5):333--348, 1986.

\bibitem{Diaconis_Aldous_86}
David Aldous and Persi Diaconis.
\newblock Shuffling cards and stopping times.
\newblock {\em Am. Math. Mon.}, 93:333--348, 1986.

\bibitem{bahadoran2011hydrodynamics}
C.~Bahadoran.
\newblock Hydrodynamics and hydrostatics for a class of asymmetric particle
  systems with open boundaries.
\newblock {\em Comm. Math. Phys.}, 310(1):1--24, 2012.

\bibitem{BGRS_euler}
C.~Bahadoran, H.~Guiol, K.~Ravishankar, and E.~Saada.
\newblock Euler hydrodynamics for attractive particle systems in random
  environment.
\newblock {\em Ann. Inst. Henri Poincar{\'e}, Probab. Stat.}, 50(2):403--424,
  2014.

\bibitem{BGRT_constructive_euler_review}
Christophe Bahadoran, Herv{\'e} Guiol, Krishnamurthi Ravishankar, and Ellen
  Saada.
\newblock Constructive {Euler} hydrodynamics for one-dimensional attractive
  particle systems.
\newblock In {\em Sojourns in probability theory and statistical physics. III.
  Interacting particle systems and random walks, a festschrift for Charles M.
  Newman}, pages 43--89. Singapore: Springer; Shanghai: NYU Shanghai, 2019.

\bibitem{benjamini2002mixingtimesbiasedcard}
I.~Benjamini, N.~Berger, C.~Hoffman, and E.~Mossel.
\newblock Mixing times of the biased card shuffling and the asymmetric
  exclusion process.
\newblock {\em Trans. Amer. Math. Soc.}, 357(8):3013--3029, 2005.

\bibitem{carinci_giardina_GR_duality_for_stochastic}
Gioia Carinci, Cristian Giardin{\`a}, Claudio Giberti, and Frank Redig.
\newblock Duality for stochastic models of transport.
\newblock {\em J. Stat. Phys.}, 152(4):657--697, 2013.

\bibitem{CGR_exact_formula}
Gioia Carinci, Cristian Giardin{\`a}, and Frank Redig.
\newblock Exact formulas for two interacting particles and applications in
  particle systems with duality.
\newblock {\em Ann. Appl. Probab.}, 30(4):1934--1970, 2020.

\bibitem{Chen_Sau_higher_HL}
Joe~P. Chen and Federico Sau.
\newblock Higher-order hydrodynamics and equilibrium fluctuations of
  interacting particle systems.
\newblock {\em Markov Process. Relat. Fields}, 27(3):339--380, 2021.

\bibitem{misanthrope}
Christiane Cocozza-Thivent.
\newblock Processus des misanthropes.
\newblock {\em Z. Wahrscheinlichkeitstheor. Verw. Geb.}, 70:509--523, 1985.

\bibitem{de_masi_ferrari_lebowitz}
A.~de~Masi, P.~A. Ferrari, and J.~L. Lebowitz.
\newblock Reaction-diffusion equations for interacting particle systems.
\newblock {\em J. Stat. Phys.}, 44:589--644, 1986.

\bibitem{Schiavo_Portinale_Sau_HL_duality}
Lorenzo Dello~Schiavo, Lorenzo Portinale, and Federico Sau.
\newblock Scaling limits of random walks, harmonic profiles, and stationary
  nonequilibrium states in {Lipschitz} domains.
\newblock {\em Ann. Appl. Probab.}, 34(2):1789--1845, 2024.

\bibitem{Diaconis1981-yy}
P.~Diaconis and M.~Shahshahani.
\newblock Generating a random permutation with random transpositions.
\newblock {\em Z. Wahrsch. Verw. Gebiete}, 57(2):159--179, 1981.

\bibitem{DLP_glauber_complete_graph}
Jian Ding, Eyal Lubetzky, and Yuval Peres.
\newblock The mixing time evolution of {Glauber} dynamics for the mean-field
  {Ising} model.
\newblock {\em Commun. Math. Phys.}, 289(2):725--764, 2009.

\bibitem{eymard_roussignom_tordeux}
R.~Eymard, M.~Roussignol, and A.~Tordeux.
\newblock Convergence of a misanthrope process to the entropy solution of {1D}
  problems.
\newblock {\em Stochastic Processes Appl.}, 122(11):3648--3679, 2012.

\bibitem{FGD_mass_migration}
Lucie Fajfrov{\'a}, Thierry Gobron, and Ellen Saada.
\newblock Invariant measures of mass migration processes.
\newblock {\em Electron. J. Probab.}, 21:52, 2016.
\newblock Id/No 60.

\bibitem{FLoreani_redig_sau_hydrodynamcis_sep_re}
Simone Floreani, Frank Redig, and Federico Sau.
\newblock Hydrodynamics for the partial exclusion process in random
  environment.
\newblock {\em Stochastic Processes Appl.}, 142:124--158, 2021.

\bibitem{floreani_redig_sau_orthogonal_duality}
Simone Floreani, Frank Redig, and Federico Sau.
\newblock Orthogonal polynomial duality of boundary driven particle systems and
  non-equilibrium correlations.
\newblock {\em Ann. Inst. Henri Poincar{\'e}, Probab. Stat.}, 58(1):220--247,
  2022.

\bibitem{franceschini2023nonequilibriumfluctuationssepalphaopen}
C.~Franceschini, P.~Gon{\c{c}}alves, M.~Jara, and B.~Salvador.
\newblock Non-equilibrium fluctuations for {{\(\operatorname{SEP}(\alpha)\)}}
  with open boundary.
\newblock {\em Stochastic Processes Appl.}, 178:39, 2024.
\newblock Id/No 104463.

\bibitem{Bea_FRancheschini_G_HL}
C.~Franceschini, P.~Gon{\c{c}}alves, and B.~Salvador.
\newblock Hydrodynamical behavior for the symmetric simple partial exclusion
  with open boundary.
\newblock {\em Math. Phys. Anal. Geom.}, 26(2):23, 2023.
\newblock Id/No 11.

\bibitem{funaki_tsunoda}
Tadahisa Funaki and Kenkichi Tsunoda.
\newblock Motion by mean curvature from {Glauber}-{Kawasaki} dynamics.
\newblock {\em J. Stat. Phys.}, 177(2):183--208, 2019.

\bibitem{funaki}
Tadahisa Funaki, Patrick van Meurs, Sunder Sethuraman, and Kenkichi Tsunoda.
\newblock Motion by mean curvature from {Glauber}-{Kawasaki} dynamics with
  speed change.
\newblock {\em J. Stat. Phys.}, 190(3):30, 2023.
\newblock Id/No 45.

\bibitem{Glauber}
R.~J. Glauber.
\newblock Time-dependent statistics of the {Ising} model.
\newblock {\em J. Math. Phys.}, 4:294--307, 1963.

\bibitem{Gobron_saada}
Thierry Gobron and Ellen Saada.
\newblock Couplings, attractiveness and hydrodynamics for conservative particle
  systems.
\newblock {\em Ann. Inst. Henri Poincar{\'e}, Probab. Stat.}, 46(4):1132--1177,
  2010.

\bibitem{G_lecture_note_IHP}
Patr{\'{\i}}cia Gon{\c{c}}alves.
\newblock Hydrodynamics for symmetric exclusion in contact with reservoirs.
\newblock In {\em Stochastic dynamics out of equilibrium. Lecture notes from
  the IHP trimester, Institut Henri Poincar\'e (IHP), Paris, France, April --
  July, 2017}, pages 137--205. Cham: Springer, 2019.

\bibitem{arXiv:gu_yang}
Chenlin Gu and Linzhi Yang.
\newblock Relaxation to equilibrium of conservative dynamics {II}: non-gradient
  exclusion processes.
\newblock Preprint, {arXiv}:2509.20797, 2025.

\bibitem{hermon2018cutoffmeanfieldzerorangeprocess}
J.~Hermon and J.~Salez.
\newblock Cutoff for the mean-field zero-range process with bounded monotone
  rates.
\newblock {\em Ann. Probab.}, 48(2):742--759, 2020.

\bibitem{kanai2010twolane}
Masahiro Kanai.
\newblock Two-lane traffic-flow model with an exact steady-state solution.
\newblock {\em Physical Review E}, 82(6):066107, 2010.

\bibitem{KFHS_mean_curvature_ZRP_Glauber}
Perla~El Kettani, Tadahisa Funaki, Danielle Hilhorst, Hyunjoon Park, and Sunder
  Sethuraman.
\newblock Mean curvature interface limit from {Glauber}+{Zero}-range
  interacting particles.
\newblock {\em Commun. Math. Phys.}, 394(3):1173--1223, 2022.

\bibitem{kipnis1999sli}
C.~Kipnis and C.~Landim.
\newblock {\em Scaling limits of interacting particle systems}, volume 320 of
  {\em Grundlehren der mathematischen Wissenschaften}.
\newblock Springer-Verlag, Berlin, 1999.

\bibitem{Lacoin_2016}
H.~Lacoin.
\newblock The cutoff profile for the simple exclusion process on the circle.
\newblock {\em Ann. Probab.}, 44(5):3399--3430, 2016.

\bibitem{Lacoin2013MixingTA}
H.~Lacoin.
\newblock Mixing time and cutoff for the adjacent transposition shuffle and the
  simple exclusion.
\newblock {\em Ann. Probab.}, 44(2):1426--1487, 2016.

\bibitem{lacoin2016simpleexclusionprocesscircle}
H.~Lacoin.
\newblock The simple exclusion process on the circle has a diffusive cutoff
  window.
\newblock {\em Ann. Inst. Henri Poincar\'e{} Probab. Stat.}, 53(3):1402--1437,
  2017.

\bibitem{Lacoin_Leblond}
Hubert Lacoin and R{\'e}mi Leblond.
\newblock Cutoff phenomenon for the simple exclusion process on the complete
  graph.
\newblock {\em ALEA, Lat. Am. J. Probab. Math. Stat.}, 8:285--301, 2011.

\bibitem{LP18}
G{\"u}nter Last and Mathew Penrose.
\newblock {\em Lectures on the {P}oisson {P}rocess}.
\newblock Institute of Mathematical Statistics Textbooks. Cambridge University
  Press, 2018.

\bibitem{Levin_2016}
D.~A. Levin and Y.~Peres.
\newblock Mixing of the exclusion process with small bias.
\newblock {\em J. Stat. Phys.}, 165(6):1036--1050, 2016.

\bibitem{LP}
D.~A. Levin and Y.~Peres.
\newblock {\em Markov chains and mixing times}.
\newblock American Mathematical Society, Providence, RI, 2017.

\bibitem{LLP_glaubet_complete}
David~A. Levin, Malwina~J. Luczak, and Yuval Peres.
\newblock Glauber dynamics for the mean-field {Ising} model: cut-off, critical
  power law, and metastability.
\newblock {\em Probab. Theory Relat. Fields}, 146(1-2):223--265, 2010.

\bibitem{Lubetzky_Sly}
E.~Lubetzky and A.~Sly.
\newblock Cutoff for the {I}sing model on the lattice.
\newblock {\em Invent. Math.}, 191(3):719--755, 2013.

\bibitem{univ_ising_cutoff}
E.~Lubetzky and A.~Sly.
\newblock Universality of cutoff for the {Ising} model.
\newblock {\em Ann. Probab.}, 45(6A):3664--3696, 2017.

\bibitem{LS_Ising_polynomial_time}
Eyal Lubetzky and Allan Sly.
\newblock Critical {Ising} on the square lattice mixes in polynomial time.
\newblock {\em Commun. Math. Phys.}, 313(3):815--836, 2012.

\bibitem{LS_cutoff_general_spin}
Eyal Lubetzky and Allan Sly.
\newblock Cutoff for general spin systems with arbitrary boundary conditions.
\newblock {\em Commun. Pure Appl. Math.}, 67(6):982--1027, 2014.

\bibitem{LS_exposition}
Eyal Lubetzky and Allan Sly.
\newblock An exposition to information percolation for the {Ising} model.
\newblock {\em Ann. Fac. Sci. Toulouse, Math. (6)}, 24(4):745--761, 2015.

\bibitem{LS16}
Eyal Lubetzky and Allan Sly.
\newblock Information percolation and cutoff for the stochastic {Ising} model.
\newblock {\em J. Am. Math. Soc.}, 29(3):729--774, 2016.

\bibitem{Martinelli}
Fabio Martinelli.
\newblock Lectures on {Glauber} dynamics for discrete spin models.
\newblock In {\em Lectures on probability theory and statistics. Ecole d'et\'e
  de Probabilit\'es de Saint-Flour XXVII--1997, Saint-Flour, France, July
  7--23, 1997}, pages 93--191. Berlin: Springer, 1999.

\bibitem{MS19}
M.~Merle and J.~Salez.
\newblock Cutoff for the mean-field zero-range process.
\newblock {\em Annals of Probability}, 47(5):3170--3201, 2019.

\bibitem{miller_peres}
Jason Miller and Yuval Peres.
\newblock Uniformity of the uncovered set of random walk and cutoff for
  lamplighter chains.
\newblock {\em Ann. Probab.}, 40(2):535--577, 2012.

\bibitem{Nam_sly}
Danny Nam and Allan Sly.
\newblock Cutoff for the {Swendsen}-{Wang} dynamics on the lattice.
\newblock {\em Ann. Probab.}, 47(6):3705--3761, 2019.

\bibitem{Propp_wilson96}
James Propp and David Wilson.
\newblock Coupling from the past: {A} user's guide.
\newblock In {\em Microsurveys in discrete probability. DIMACS workshop,
  Princeton, NJ, USA, June 2--6, 1997}, pages 181--192. Providence, RI: AMS,
  American Mathematical Society, 1998.

\bibitem{Rezakhanlou1991HydrodynamicLF}
F.~Rezakhanlou.
\newblock Hydrodynamic limit for attractive particle systems on {${\bf Z}^d$}.
\newblock {\em Comm. Math. Phys.}, 140(3):417--448, 1991.

\bibitem{Rezakhanlou1995MicroscopicSO}
F.~Rezakhanlou.
\newblock Microscopic structure of shocks in one conservation laws.
\newblock {\em Ann. Inst. H. Poincar\'e{} C Anal. Non Lin\'eaire},
  12(2):119--153, 1995.

\bibitem{salez_universality_cutoff}
J.~Salez.
\newblock Universality of cutoff for exclusion with reservoirs.
\newblock {\em Ann. Probab.}, 51(2):478--494, 2023.

\bibitem{schutz-sandow}
Gunter Sch\"utz and Sven Sandow.
\newblock Non-abelian symmetries of stochastic processes: Derivation of
  correlation functions for random-vertex models and
  disordered-interacting-particle systems.
\newblock {\em Phys. Rev. E}, 49:2726--2741, 04 1994.

\bibitem{tanaka_tsunoda}
Ryokichi Tanaka and Kenkichi Tsunoda.
\newblock Glauber-exclusion dynamics: rapid mixing regime.
\newblock {\em Electron. J. Probab.}, 27:26, 2022.
\newblock Id/No 141.

\bibitem{tordeux2010etude}
Antoine Tordeux.
\newblock {\em {\'E}tude de processus en temps continu mod{\'e}lisant
  l'{\'e}coulement de flux de trafic routier}.
\newblock PhD thesis, Universit{\'e} Paris-Est, France, 06 2010.
\newblock Available at \url{https://theses.fr/2010PEST1043}.

\bibitem{tran2022meanfieldzerorangeprocessunbounded}
H.~Q. Tran.
\newblock The mean-field zero-range process with unbounded monotone rates:
  mixing time, cutoff, and {P}oincar\'e{} constant.
\newblock {\em Ann. Appl. Probab.}, 33(3):1732--1757, 2023.

\bibitem{HQ}
H.Q Tran.
\newblock Cutoff for the glauber-exclusion process in the full high-temperature
  regime: an information percolation approach.
\newblock {\em arXiv preprint 2310.15810}, 2023.

\bibitem{Tsegaye}
Eyob Tsegaye.
\newblock Mixing time and cutoff for the {{\(k\)}}-{SPEP}.
\newblock {\em Stochastic Processes Appl.}, 191:19, 2026.
\newblock Id/No 104776.

\bibitem{tsunoda}
Kenkichi Tsunoda.
\newblock Exponentially slow mixing and hitting times of rare events for a
  reaction-diffusion model.
\newblock {\em ALEA, Lat. Am. J. Probab. Math. Stat.}, 19(2):1161--1184, 2022.

\bibitem{Wilson_2004}
D.~B. Wilson.
\newblock Mixing times of {L}ozenge tiling and card shuffling {M}arkov chains.
\newblock {\em Ann. Appl. Probab.}, 14(1):274--325, 2004.

\bibitem{Xu_Zhao}
Lu~Xu and Linjie Zhao.
\newblock Hydrodynamics for asymmetric simple exclusion on a finite segment
  with {Glauber}-type source.
\newblock {\em J. Stat. Phys.}, 191(7):30, 2024.
\newblock Id/No 78.

\end{thebibliography}
\end{document}